\documentclass[final]{siamart0516}

\usepackage{lipsum}
\usepackage{amsfonts}
\usepackage{graphicx}
\usepackage{epstopdf}
\usepackage{algorithmic}

\input{mysymbol.sty}

\ifpdf
  \DeclareGraphicsExtensions{.eps,.pdf,.png,.jpg}
\else
  \DeclareGraphicsExtensions{.eps}
\fi

\newcommand{\TheTitle}{{Surpassing Gradient Descent Provably: A Cyclic Incremental Method with Linear Convergence Rate}} 
\newcommand{\TheAuthors}{Aryan Mokhtari, Mert G{\"u}rb{\"u}zbalaban, and Alejandro Ribeiro}

\headers{Surpassing Gradient Descent Provably}{\TheAuthors}

\title{{\TheTitle}\thanks{{This work was supported by NSF CAREER CCF-0952867 and ONR N00014-12-1-0997. This paper expands the results and presents proofs that are referenced in \cite{7953047}.}}}

\author{
  Aryan Mokhtari\thanks{Department of Electrical and Systems Engineering, University of Pennsylvania, Philadelphia, PA
    (\email{aryanm@seas.upenn.edu},\email{aribeiro@seas.upenn.edu})}
  \and
  Mert G{\"u}rb{\"u}zbalaban\thanks{Department of Management Science and Information Systems, Rutgers University, Piscataway, NJ (\email{mgurbuzbalaban@business.rutgers.edu}).}
  \and
  Alejandro Ribeiro\footnotemark[2]
}

\usepackage{amsopn}

\newtheorem{assumption}{\hspace{0pt}\bf Assumption}
\newtheorem{remark}{\hspace{0pt}\it Remark}

\begin{document}

\maketitle

\begin{abstract}
Recently, there has been growing interest in developing optimization methods for solving large-scale machine learning problems. Most of these problems boil down to the problem of minimizing an average of a finite set of smooth and strongly convex functions where the number of functions $n$ is large. Gradient descent method (GD) is successful in minimizing convex problems at a fast linear rate; however, it is not applicable to the considered large-scale optimization setting because of the high computational complexity. Incremental methods resolve this drawback of gradient methods by replacing the required gradient for the descent direction with an incremental gradient approximation. They operate by evaluating one gradient per iteration and executing the average of the $n$ available gradients as a gradient approximate. Although, incremental methods reduce the computational cost of GD, their convergence rates do not justify their advantage relative to GD in terms of the total number of gradient evaluations until convergence. In this paper, we introduce a Double Incremental Aggregated Gradient method (DIAG) that computes the gradient of only one function at each iteration, which is chosen based on a cyclic scheme, and uses the aggregated average gradient of all the functions to approximate the full gradient. The iterates of the proposed DIAG method uses averages of both iterates and gradients in oppose to classic incremental methods that utilize gradient averages but do {\it not} utilize iterate averages. We prove that not only the proposed DIAG method converges linearly to the optimal solution, but also its linear convergence factor justifies the advantage of incremental methods on GD. In particular, we prove that the worst case performance of DIAG is better than the worst case performance of GD. Numerical experiments on quadratic programming and logistic regression problems showcase the advantage of DIAG relative to GD and other incremental methods.
\end{abstract}


\begin{AMS}
{  90C06, 90C25, 90C30, 90C52}
\end{AMS}


\section{Introduction}\label{sec_Introduction}
This paper focuses on finite sum optimization where the objective function can be written as the sum of a set of strongly convex functions. In particular, consider $\bbx\in \reals^p$ as the optimization variable and $f_i:\reals^p\to\reals$ as the $i$-th available function. We aim to find the minimizer of the average function $f(\bbx)=({1}/{n}) \sum_{i=1}^n f_{i} (\bbx)$, i.e., we intend to solve the optimization problem
\begin{equation}\label{org_prob}
\bbx^* = \argmin_{\bbx\in \reals^p} f(\bbx) := \argmin_{\bbx\in \reals^p} \frac{1}{n} \sum_{i=1}^n f_{i} (\bbx).
\end{equation}
In this paper, we refer to $f_i$ as the instantaneous functions and the average function $f$ as the global objective function. This class of optimization problems arises in many fields such as machine learning \cite{BottouCun,bottou2010large,SS,cevher2014convex}, optimal control,
\cite{Bullo2009,Cao2013-TII,LopesEtal8}, and wireless communications
\cite{Ribeiro10, Ribeiro12}. Our focus is on problems where the instantaneous functions $f_i$ are smooth and strongly convex. 

To explain the contribution of this paper we have to discuss the rate and constants that characterize convergence of the different first order methods that can be used to solve the problem in \eqref{org_prob}. To begin with we can neglect the specific form of $f$ and use the conventional gradient descent (GD) method which is known to converge linearly to the optimal argument \cite{nesterov2004introductory}. This linear convergence rate comes from using individual iterations that are very costly when the number of functions $n$ is large and motivates the use of stochastic and incremental methods in which only one of the instantaneous gradients $\nabla f_{i}$ is evaluated at each iteration. The selection is random in stochastic methods and cyclic in incremental methods. In either case the idea is that individual iterations are less efficient but since $n$ stochastic or incremental operations have the same cost as one GD iteration, overall convergence is faster.

Although faster convergence is observed in many practical situations, it is not known if it is possible to design a stochastic or incremental method with convergence guarantees that are better than the convergence guarantees of GD. The stochastic gradient descent (SGD) method \cite{robbins1951stochastic,bottou2010large}, for instance, is known to have a sublinear convergence rate and is therefore surpassed by regular GD as the number of iterations grows. This limitation is in fact the motivation for alternative stochastic descent methods that achieve linear convergence rates by  reducing the variance of stochastic descent directions. Examples of this growing and consequential literature includes stochastic averaging gradient algorithms \cite{roux2012stochastic, defazio2014saga, defazio2014finito,mokhtari2016dsa}, variance reduction methods \cite{johnson2013accelerating, xiao2014proximal}, dual coordinate methods \cite{ shalev2013stochastic, shalev2016accelerated}, hybrid algorithms \cite{zhang2013linear, konevcny2013semi}, and majorization-minimization algorithms \cite{mairal2015incremental}. All of these stochastic methods are successful in achieving a linear convergence rate in expectation with individual iterations that have cost comparable to the cost of SGD iterations. However, the linear convergence constants of these methods are not necessarily better than the linear convergence constant of GD for a problem with comparable condition number. This leaves open the possibility that the worst case performance of these methods is worse than the worst case performance of GD -- see Section \ref{sec_prel}.

Given that the only difference between stochastic and incremental methods is that in the latter functions are chosen in a \textit{cyclic} order -- as opposed from the selection in stochastic methods which is uniformly at random -- it is not surprising that analogous statements can be made for incremental gradient descent methods (IGD) \cite{bertsekas1997new, tseng1998incremental, nedic2001incremental, rabbat2005quantized,blatt2007convergent,ram2009incremental,johansson2009randomized, bertsekas2011incremental, tseng2014incrementally, gurbuzbalaban2015convergence, vanli2016global}. Standard IGD has a slow sublinear convergence rate, which motivates the  introduction of memory. This is done in the definition of the incremental aggregated gradient (IAG) method that is shown to achieve linear convergence \cite{gurbuzbalaban2015convergence} but with a constant that is not necessarily better that the GD constant. Thus, and as in the case of stochastic methods, it is possible that the worst case performance of IAG is worse than the worst case performance of GD -- see Section \ref{sec_prel}.

The contribution of this paper is to introduce a first order incremental method that has a linear convergence rate with a constant that is better than the GD constant of a problem with comparable condition number. This means that the worst case performance of the proposed algorithm is guaranteed to be no worse than the worst case performance of GD. The algorithm relies on keeping memory of past variable {\it and} gradient evaluations and is therefore termed the Double Incremental Aggregated Gradient (DIAG) method to emphasize the difference with regular IAG methods in which only gradient histories are maintained. This major difference comes from the fact that DIAG uses a different approximation of the global function $f$ at each iteration from the one used in IAG. In particular, DIAG approximates each instantaneous function $f_i$ by the sum of its first order approximation and a proximity term, both evaluated with respect to the same iterate, whereas IAG uses different points for the first-order approximation and the proximity condition. We show that this critical difference leads to an incremental algorithm with a smaller linear convergence factor. Moreover, the linear convergence factor of the proposed DIAG method justifies the use of incremental methods to improve the performance of GD. In particular, we show that the worse case scenario of DIAG is guaranteed to be better than the worse case scenario of GD. Based on our knowledge, this is the first incremental method which is guaranteed to improve the worse case performance of GD.

We start the paper by presenting the GD and IAG methods and studying their convergence guarantees for the case that the instantaneous functions $f_i$ are strongly convex and their gradients $\nabla f_i$ are Lipschitz continuous (Section \ref{sec_prel}). We clarify the reason that the convergence analysis of IAG cannot guarantee the advantage of incremental methods with respect to GD. Then, we present the proposed DIAG method which uses both variable and gradient averages in oppose to IAG that only uses aggregated gradient average (Section \ref{sec_algorithm}). We explain the intuition behind this difference by comparing the function approximations used in these methods. Further, we suggest an efficient mechanism to implement the proposed DIAG algorithm that has the computational complexity of the order $O(p)$ which is significantly lower than of GD given by $\mathcal{O}(np)$ (Section \ref{impl_sec}). Further, we explain the connection between the proposed DIAG method and the majorization-minimization method (MISO) proposed in \cite{mairal2015incremental}, and highlight the differences between these two algorithms (Remark \ref{remark:MISO}).

The convergence analysis of the DIAG method is then presented (Section \ref{sec_Convergence}). We first prove a fundamental lemma that shows the error of DIAG at each iteration is strictly smaller than the average of the errors of the last $n$ iterations (Lemma \ref{fundamental_lemma}). We use this result to prove that the sequence of variables generated by DIAG converges to the optimal argument $\bbx^*$ (Proposition \ref{CIAG_better_than_GD_lemma}), and, in particular, the convergence rate of the iterates evaluated after each pass over the dataset is linear (Corollary \ref{imp_cor}). This linear convergence factor guarantees that one pass of DIAG is more efficient than one iteration of gradient descent, i.e., the upper bound for the error of DIAG after $n$ gradient evaluations is strictly smaller than the one for GD. Then, we prove that the whole sequence of DIAG iterates is linearly convergent (Theorem \ref{thm_lin}) and characterize the linear convergence factor (Theorem \ref{thm_best_lin}). We extend our convergence results by studying the worst-case asymptotic rate of DIAG (Section \ref{sec:worst_case}). We use the Perron-Frobenius (PF) theory to show that an upper bound for the sequence of DIAG errors has an asymptotic linear convergence rate which is strictly better than the linear convergence factor of GD (Theorem \ref{theorem-rate-bounds}). 

We compare the performances of DIAG, GD, and IAG in solving a quadratic programming and a binary classification problem (Section \ref{sec:simulations}). Numerical results for the quadratic programming confirm that DIAG outperforms GD. In particular, the relative performance of DIAG and GD does not vary by changing the problem condition number, while IAG is not preferable to GD when the problem condition number is relatively large.  Moreover, DIAG outperforms IAG irrespective to the problem parameters. The convergence paths of these methods for the binary classification problem, which is a logistic regression minimization, confirm the observations for the quadratic programming problem. Finally, we close the paper by concluding remarks (Section \ref{sec:conclusions}).

\subsection{Notation}Vectors are written as $\bbx\in\reals^p$ and matrices as $\bbA\in\reals^{p\times p}$.
Given $n$ vectors $\bbx_i$, the vector $\bbx=[\bbx_1;\ldots;\bbx_n]$ represents a stacking of the
elements of each individual $\bbx_i$. We use $\|\bbx\|$ and  $\|\bbA\|$ to denote the Euclidean norm of vector $\bbx$ and matrix $\bbA$, respectively. Given a function $f$ its gradient $\bbx$ is denoted as $\nabla f(\bbx)$. 

\section{Related Works and Preliminaries}\label{sec_prel}

Since the objective function in \eqref{org_prob} is convex, descent methods can be used to find the optimal argument $\bbx^*$. In this paper, we are interested in studying methods that converge to the optimal argument of the global objective function $f$ at a linear rate. It is customary for the linear convergence analysis of first-order methods to assume that the functions are smooth and strongly convex. We formalize these conditions in the following assumption.

\begin{assumption} \label{cnvx_lip}
The functions $f_i$ are differentiable and strongly convex with constant $\mu>0$, i.e., for all $\bbx,\bby\in \reals^p$ we can write 
\begin{equation}
(\nabla f_{i}(\bbx)-\nabla f_{i}(\bby))^T(\bbx-\bby)
\geq 
\mu\|\bbx-\bby\|^2.
\end{equation} 
Moreover, the gradients $\nabla f_i$ are Lipschitz continuous with constant $L<\infty$, i.e., for all $\bbx,\bby\in \reals^p$ we have
\begin{equation}
\|\nabla f_{i}(\bbx)-\nabla f_{i}(\bby)\|
\leq 
L\|\bbx-\bby\|.
\end{equation} 
\end{assumption}

The strong convexity of the functions $f_i$ with constant $\mu$ implies that the global objective function $f$ is also strongly convex with constant $\mu$. Likewise, the Lipschitz continuity of the gradients $\nabla f_i$ with constant $L$ yields Lipschitz continuity of the global objective function gradients $\nabla f$ with constant $L$. Note that the conditions in Assumption \ref{cnvx_lip} are mild and hold for most large-scale machine learning applications such as, linear regression, logistic regression, least squares, and support vector machines.

The optimization problem in \eqref{org_prob} can be solved using the gradient descent (GD) method \cite{nesterov2004introductory}. The idea of GD is to update the current iterate $\bbx^k$ by descending through the negative direction of the current gradient $\nabla f(\bbx^k)$ with a proper stepsize $\eps^k$. In other words, the update of GD for solving problem \eqref{org_prob} at step $k$ is defined as 
\begin{equation}\label{GD_update}
\bbx^{k+1}=\bbx^k-\eps^k\ \!  \nabla f(\bbx^k)=\bbx^k-\frac{\eps^k}{n} \sum_{i=1}^n  \nabla f_i(\bbx^k).
\end{equation}
Convergence analysis of GD in \cite{nesterov2004introductory} shows that the sequence of iterates $\bbx^k$ converges linearly to the optimal argument $\bbx^*$ if the stepsize is constant and satisfies $\eps^k=\eps<2/L$. The fastest convergence rate is achieved by the stepsize $\eps=2/(\mu+L)$ which leads to the linear convergence factor $(\kappa-1)/(\kappa+1)$, i.e., 
\begin{equation}\label{GD_rate}
\|\bbx^{k}-\bbx^*\|  \leq \left( \frac{\kappa-1}{\kappa+1}\right)^k\ \! \|\bbx^0-\bbx^*\|,
\end{equation}
where $\kappa:=L/\mu$ is the global objective function condition number. Although, GD has a fast linear convergence rate, it is not computationally affordable in large-scale applications because of its high computational complexity. To comprehend this limitation, note that each iteration of GD requires $n$ gradient evaluations which is not computationally affordable in large-scale applications with massive values of $n$. Stochastic gradient descent (SGD) arises as a natural solution in large-scale settings. SGD modifies the update of GD by approximating the gradient of the global objective function $\nabla f$ by the average of a small number of instantaneous gradients chosen uniformly at random from the set of $n$ gradients. To be more precise, the update of SGD at step $k$ is defined as 
\begin{equation}\label{SGD_update}
\bbx^{k+1}=\bbx^k-\frac{\eps^k}{b} \sum_{i\in \ccalS_b^k} \nabla f_i(\bbx^k),
\end{equation}
where $\ccalS_b^k$ is defined as a random set that contains the indices of $b$ functions that are chosen for the update SGD at step $k$. Note that the components of the set $\ccalS_b^k$ are chosen uniformly at random from the set of indices $\{1,2,\dots,n\}$. Since the stochastic gradient $(1/{b} )\sum_{i\in \ccalS_b^k} \nabla f_i(\bbx^k)$ is an unbiased estimator of the gradient $\nabla f(\bbx^k)=(1/n)\sum_{i=1}^n \nabla f_i(\bbx^k)$, the sequence of the iterates generated by SGD converges to the optimal argument in expectation. However, the convergence rate is sublinear and slower than the linear convergence of GD. In particular, the expected error $\|\bbx^{k}-\bbx^*\|^2$ of SGD is bounded above as
$\E{\|\bbx^{k}-\bbx^*\|^2}  \leq \mathcal{O}\left({1}/{k}\right),$
for a diminishing stepsizes $\eps^k$ of the order $1/k$. It is worth mentioning  that the expectation is taken with respect to the indices of the chosen random functions up to step $k$. 

One may use a cyclic order instead of stochastic selection of functions in SGD which leads to the update of incremental gradient descent method (IGD) as in \cite{blatt2007convergent,tseng2014incrementally}. Similar to the case for SGD, the sequence of iterates generated by the IGD method converges to the optimal argument at a sublinear rate of the order $\mathcal{O}\left({1}/{k}\right)$ when the stepsize is diminishing. SGD and IGD are able to reduce the computational complexity of GD by requiring only one gradient evaluation per iteration; however, they both suffer from slow (sublinear) convergence rates. 

The sublinear convergence rate of SGD has been improved recently by the stochastic average gradient method (SAG) which also can be interpreted as a stochastic incremental \textit{aggregated} gradient method. The SAG method updates only one gradient per iteration and uses the average of the most recent version of all gradients --  gradients of all functions $f_1$, \dots, $f_n$ -- as an approximation for the full gradient \cite{roux2012stochastic}. To be more specific, define $\bby_{i}^k$ as the copy of the decision variable $\bbx$ for the last time that the function $f_i$'s gradient is updated. {In other words, the variable $\bby_i^k$ is updated as 
\begin{equation}\label{def_y_new_review}
\bby_i^{k+1}=
\begin{cases}
\bbx^{k+1} \qquad  \text{if} \ i=i^k,\\
\bby_i^{k} \qquad\quad   \text{otherwise},
\end{cases}
\end{equation}
where $i^k$ is the index of the function chosen at step $k$.} Note that in the SAG method the random index $i^k$ is chosen uniformly at random and the gradient of its corresponding function $\nabla f_{i^k} (\bbx^k)$ is evaluated and stored as $\nabla f_{i^k} (\bby_i^k)$. Then, the update of SAG at step $k$ is given by 
\begin{equation}\label{SAG_update}
\bbx^{k+1}= \bbx^k - \frac{\eps}{n} \sum_{i=1}^n \nabla f_i (\bby_i^k),
\end{equation}
%
which uses the gradients of all the $n$ functions evaluated at different time steps. 
{The sequence of iterates generated by SAG converges linearly to $\bbx^*$ in expectation with respect to the choices of random indices, i.e., 
\begin{equation}\label{SAGA_lin_convg}
\E{\|\bbx^k-\bbx^*\|^2}\leq \left(1-\min\left\{\frac{1}{16\kappa},\frac{1}{8n}\right\}\right)^k C_0,
\end{equation}
where $C_0$ is a constant independent of $n$ and $\kappa$ \cite{schmidt2017minimizing}. However, the linear convergence constant of SAG in \eqref{SAGA_lin_convg} is not necessarily better than the linear convergence constant of GD for a problem with comparable condition number. To be more precise, the residual $\|\bbx^k-\bbx^*\|$ of SAG in expectation decays by the factor of $(1-\min\left\{\frac{1}{16\kappa},\frac{1}{8n}\right\})^{n/2}$ after a pass over the dataset which might not be better than the upper bound for the residual of GD that decays with the factor of $(\kappa-1)/(\kappa+1)$. As an example, for the problem that $n=100$ and $\kappa=10$, the worst case performance of GD after $m$ passes over the set of functions ($m$ iterations) is bounded above by $((\kappa-1)/(\kappa+1))^m \|\bbx^0-\bbx^*\| \approx 0.8181^m \|\bbx^0-\bbx^*\|$, while the worst performance of SAG after $m$ passes over the set of functions ($mn$ iterations) is bounded above by $C_0^{1/2} (1-\min\left\{\frac{1}{16\kappa},\frac{1}{8n}\right\})^{nm/2} \approx 0.9393^m C_0^{1/2}$. Note that the constants $\|\bbx^0-\bbx^*\|$ and $C_0$ are negligible for sufficiently large $m$.} 
Similar examples can be derived for other first-order stochastic methods with linear convergence rate in \cite{defazio2014saga, defazio2014finito,johnson2013accelerating,xiao2014proximal,shalev2013stochastic, shalev2016accelerated,zhang2013linear, konevcny2013semi,mairal2015incremental}. Beside this issue, the results for all these stochastic first-order methods hold in \textit{expectation}. Thus, there is a positive probability that the sequence of iterates generated by these methods might not converge at a linear rate.

The other alternative for solving the optimization problem in \eqref{org_prob} is the Incremental Aggregated Gradient (IAG) method which is a middle ground between GD and IGD. The IAG method requires one gradient evaluation per iteration, as in IG, while it approximates the gradient of the global objective function $\nabla f(\bbx)$ by the average of the most recent gradient of all instantaneous functions \cite{blatt2007convergent}, and it has a linear convergence rate, as in GD. In the IAG method, the functions are chosen in a cyclic order and it takes $n$ iterations to have a pass over all the available functions. To introduce the update of IAG, recall the definition of $\bby_{i}^k$ as the copy of the decision variable $\bbx$ for the last time that the function $f_i$'s gradient is updated before step $k$ which can be updated as in \eqref{def_y_new_review}. Then, the update of IAG is given by 
\begin{equation}\label{IAG_update}
\bbx^{k+1}= \bbx^k - \frac{\eps}{n} \sum_{i=1}^n \nabla f_i (\bby_i^k),
\end{equation}
which is identical to the update of SAG in \eqref{SAG_update}, and the only difference is in the scheme that the index $i^k$ is chosen.

The convergence results in \cite{tseng2014incrementally} provide global
convergence and local linear convergence of IAG in a more general setting when each component function satisfies a local Lipschitzian error condition. More recently, a new convergence analysis of IAG has been studied in \cite{gurbuzbalaban2015convergence} which shows global linear convergence of IAG for strongly convex functions with Lipschitz continuous gradients. In particular, it has been shown that the sequence of iterates $\bbx^k$ generated by IAG satisfies the following inequality
\begin{equation}\label{IAG_rate}
\|\bbx^{k}-\bbx^*\| \leq  \left(1-\frac{2}{25n(2n+1)(\kappa+1)^2}\right)^k  \|\bbx^{0}-\bbx^*\|.
\end{equation}
Notice that the convergence rate of IAG is linear and eventually the error of IAG will be smaller than the errors of SGD and IGD which diminish with a sublinear rate of $\mathcal{O}(1/k)$. To compare the performance of GD and IAG it is fair to compare one iteration of GD with $n$ iterations of IAG. This is reasonable since one iteration of GD requires $n$ gradient evaluations, while IAG uses $n$ gradient evaluations after $n$ iterations.  Comparing the decrement factors of GD in \eqref{GD_rate} and IAG after $n$ gradient evaluations in \eqref{IAG_rate} shows that there is no guarantee that IAG is preferable to GD for all choices of condition number $\kappa$ and number of functions $n$, since we could face the scenario that 
\begin{equation}\label{GD_IAG_rate_comp}
 \left( \frac{\kappa-1}{\kappa+1}\right)< \left(1-\frac{2}{25n(2n+1)(\kappa+1)^2}\right)^n.
\end{equation}
As an example, for the problem with $n=\kappa=100$, the inequality in \eqref{GD_IAG_rate_comp} holds and the worst case performance of IAG is worse than the one for GD. Note that the bound for GD in \eqref{GD_rate} is strict and we can design a sequence which satisfies the equality case of the result in \eqref{GD_rate}{\footnote{Consider the quadratic programming $f(\bbx)=(1/2)\bbx^T\bbA\bbx$, where $\bbA=\diag[\mu,L]$ which has the optimal argument $\bbx^*=\bb0\in \reals^2$. Then, by setting $\eps=2/(\mu+L)$ the sequence of iterates generated by GD satisfies the relation $\|\bbx^{m}-\bbx^*\|= \rho^m\|\bbx^{0}-\bbx^*\|$.}}. However, the bound in \eqref{IAG_rate} is not necessarily tight and it could be the reason that the comparison in \eqref{GD_IAG_rate_comp} does not justify the use of IAG instead of GD. Our goal in this paper is to come up with a first-order incremental method that has a guaranteed upper bound which is better than the one for GD in \eqref{GD_rate}. We propose this algorithm in the following section.

\section{Algorithm Definition}\label{sec_algorithm}

In this section, we propose a novel incremental gradient method that unlike other incremental methods is able to improve upon the worst case performance of GD. To do so, we first introduce a new interpretation of the IAG method. Recall the definition of the variable $\bby_i^k$ as the copy of the decision variable $\bbx$ for the last time that function $f_i$ is chosen for gradient update and its update scheme in \eqref{def_y_new_review}. The update of IAG in \eqref{IAG_update} can be interpreted as the solution of the optimization program 
\begin{equation}
\bbx^{k+1}=\argmin_{\bbx\in \reals^p} \left\{ 
 \frac{1}{n}\sum_{i=1}^n f_i (\bby_i^k) + \frac{1}{n}\sum_{i=1}^n\nabla f_i (\bby_i^k)^T (\bbx-\bby_i^k) 
	+ \frac{1}{n}\sum_{i=1}^n\frac{1}{2\eps}  \|\bbx-\bbx^k\|^2
\right\}.
\end{equation}
This interpretation shows that in the update of IAG each instantaneous function $f_{i}(\bbx)$ is approximated by the following approximation 
\begin{equation}\label{IAG_approx}
f_i (\bbx) \approx f_i (\bby_i^k) +\nabla f_i (\bby_i^k)^T (\bbx-\bby_i^k) 
	+\frac{1}{2\eps}  \|\bbx-\bbx^k\|^2 .
\end{equation}
Notice that the first two terms $f_i (\bby_i^k) +\nabla f_i (\bby_i^k)^T (\bbx-\bby_i^k) $ correspond to the first order approximation of the function $f_i$ around the iterate $\bby_i^k$. The last term which is ${1}{/(2\eps)}  \|\bbx-\bbx^k\|^2$ is a proximal term that is added to the first order approximation. This approximation is different from the classic approximation that is used in first-order methods, since the first-order approximation is evaluated around the point $\bby_{i}^k$ which is different from the iterate $\bbx^k$ used in the proximal term. This observation verifies that the IAG algorithm performs well when the delayed variables $\bby_i^k$ are close to the current iterate $\bbx^k$ which is true when the stepsize $\eps$ is very small or the iterates are all close to the optimal solution. 

We resolve this issue by introducing a different approach for approximating each component function $f_i$, In particular, we use the approximation
\begin{equation}\label{IAG_2_approx}
f_i (\bbx) \approx f_i (\bby_i^k) +\nabla f_i (\bby_i^k)^T (\bbx-\bby_i^k) 
	+\frac{1}{2\eps}  \|\bbx-\bby_i^k\|^2.
\end{equation}
As we observe, the approximation in \eqref{IAG_2_approx} is more consistent to classic first-order methods comparing to the one for IAG in \eqref{IAG_approx}. This is true since the first order approximation and the proximal term in \eqref{IAG_2_approx} are evaluated with respect to the same point $\bby_i^k$. Indeed, the approximation in \eqref{IAG_2_approx} implies that the global objective function $f(\bbx)$ can be approximated by
\begin{equation}\label{IAG_2_global_approx}
f(\bbx) \approx \frac{1}{n}\sum_{i=1}^n f_i (\bby_i^k) +\frac{1}{n} \sum_{i=1}^n\nabla f_i (\bby_i^k)^T (\bbx-\bby_i^k) 
	+\frac{1}{n} \sum_{i=1}^n\frac{1}{2\eps}  \|\bbx-\bby_i^k\|^2.
\end{equation}
We can approximate the optimal argument of the global objective function $f$ by minimizing its approximation in \eqref{IAG_2_global_approx}. Thus, the updated iterate $\bbx^{k+1}$ can be computed as the minimizer of the approximated global objective function in \eqref{IAG_2_global_approx}, i.e., 
\begin{equation}\label{IAG_2_update}
\bbx^{k+1}=\argmin_{\bbx\in \reals^p}\left\{ \frac{1}{n}\sum_{i=1}^n f_i (\bby_i^k) + \frac{1}{n}\sum_{i=1}^n\nabla f_i (\bby_i^k)^T (\bbx-\bby_i^k) 
	+ \frac{1}{n}\sum_{i=1}^n\frac{1}{2\eps}  \|\bbx-\bby_i^k\|^2 \right\}.
\end{equation}
Considering the convex programming in \eqref{IAG_2_update} we can derive a closed form expression for the update of $\bbx^{k+1}$ as
\begin{equation}\label{incremental_qn_update}
\bbx^{k+1}=
			 \frac{1}{n} \sum_{i=1}^n \bby_i^k - \frac{\eps}{n} \sum_{i=1}^n \nabla f_i(\bby_i^k) .
\end{equation}
We refer to the proposed method with the update in \eqref{incremental_qn_update} as the Double Incremental Aggregated Gradient method (DIAG). This appellation is justified considering that the update of DIAG requires the \textit{incremented aggregate} of \textit{both} variables and gradients and only uses \textit{gradient} (first-order) information. 

Notice that since we use a cyclic scheme, the set of variables $\{\bby_1^k,\bby_2^k,\dots,\bby_n^k\}$ is equal to the set $\{\bbx^k,\bbx^{k-1},\dots,\bbx^{k-n+1}\}$. 
Therefore, the iterate $\bbx^{k+1}$ is a function of the last $n$ iterates $\{\bbx^k,\bbx^{k-1},\dots,\bbx^{k-n+1}\}$. This observation has a fundamental role in the analysis of the proposed DIAG method -- see Section 4. 

\begin{remark}\label{remark:MISO}
One may consider the proposed DIAG method as a cyclic version of the stochastic methods Finito and MISO algorithms introduced in \cite{defazio2014finito} and\cite{mairal2015incremental}, respectively. This is a valid interpretation; however, the convergence analyses and guarantees of these methods are quite different. The proposed DIAG  method is designed based on the new interpretation in \eqref{IAG_2_approx} that leads to a novel proof technique -- see Lemma \ref{fundamental_lemma} -- which is different from the analysis of Finito/MISO in \cite{defazio2014finito} and \cite{mairal2015incremental}. This analytical difference leads to different convergence guarantees. In particular, the Finito/MISO algorithm cannot  improve the performance of GD for all choices of $n$ and $\kappa$, while the established theoretical results for DIAG in Section  \ref{sec_Convergence} guarantee that DIAG outperforms GD under any choices of $n$ and $\kappa$.
\end{remark}


\subsection{Implementation Details}\label{impl_sec}
Naive implementation of the update in \eqref{incremental_qn_update} requires computation of sums of $n$ vectors per iteration which is computationally costly. This unnecessary computation can be avoided by tracking the sums over time. {To be more precise, we can define $\bbv^k$ as the vector that tracks the first sum in \eqref{incremental_qn_update} which is the sum of the variables. The vector $\bbv^k$ can be updated as 
\begin{equation}\label{sum_var_update}
\bbv^{k+1}=\bbx^{k+1}-\bby_{i^k}^k + \bbv^k  ,
\end{equation}
where $i^k$ is the index of the function chosen at step $k$. Likewise, we define the vector $\bbg^k$ as the vector that tracks the sum of gradients in \eqref{incremental_qn_update}, and it can be updated as
\begin{equation}\label{sum_grad_update}
\bbg^{k+1}=\nabla f_{i^k} (\bbx^{k+1})- \nabla f_{i^k} (\bby_{i^k}^k)+\bbg^{k}.
\end{equation}
Note that the vectors $\bbv^k$ and $\bbg^k$ are initialized as $\bbv^0=n\bbx^0$ and $\bbg^0=\sum_{i=1}^n\nabla f_i(\bbx^0)$.}

%
\begin{algorithm}[t]{
\caption{Double Incremental Aggregated Gradient method (DIAG)}
\label{algo_lbfgs} 
\begin{algorithmic}[1]
   \STATE \textbf{Initialization}: $\{\bby_i^0\}_{i=1}^{i=n}=\bbx^0$, $\bbv^0=n\bbx^0$, and $\bbg^0=\sum_{i=1}^n\nabla f_i(\bbx^0)$
   \FOR   {$k= 0, 1, \ldots$ } 
   \STATE  Compute the function index $i^k=$ mod$(k,n)+1$ 
   \vspace{1mm}
      \STATE Compute  
             ${\bbx^{k+1}=
			 \frac{1}{n} \bbv^k - \frac{\eps}{n} \bbg^k .}$
   \vspace{1mm}
	\STATE Update sum of variables ${ \bbv^{k+1}=  \bbx^{k+1}-\bby_{i^k}^k}+\bbv^k	$.	 
	  \vspace{1mm}
      \STATE Compute $\nabla f_{i^k}(\bbx^{k+1})$ and update ${\bbg^{k+1}= \nabla f_{i^k} (\bbx^{k+1})- \nabla f_{i^k} (\bby_{i^k}^k)+\bbg^k .}$
        \vspace{1mm}
      \STATE Replace $\bby_{i^k}^k$ and $\nabla f_{i^k}( \bby_{i^k}^k)$ by $\bbx^{k+1}$ and $\nabla f_{i^k}(\bbx^{k+1})$, respectively. The other elements remain unchanged, i.e., $\bby_{i}^{k+1}\!=\!\bby_{i}^k$ and $\nabla f_i( \bby_{i}^{k+1})\!=\!\nabla f_i( \bby_{i}^k)$ for $i\!\neq\! {i^k}$.
   \ENDFOR
\end{algorithmic}}\end{algorithm}

The proposed double incremental aggregated gradient (DIAG) method is summarized in Algorithm 1. The variables for all the copies of the vector $\bbx$ are initialized by vector $\bbx^0$, i.e., $\bby_1^0=\dots=\bby_n^0=\bbx^0$, and their corresponding gradients are stored in the memory. At each iteration $k$, the updated variable $\bbx^{k+1}$ is computed in Step 4 using the update in \eqref{incremental_qn_update}. The sums of variables and gradients are updated in Step 5 and 6, respectively, following the recursions in \eqref{sum_var_update} and \eqref{sum_grad_update}. In Step 7, the old variable $\bby_{i^k}^k$ and gradient $\nabla f_{i^k}( \bby_{i^k}^k)$ of the updated function $f_{i^k}$ are replaced with their updated versions, i.e., $\bbx^{k+1}$ and $\nabla f_{i^k}(\bbx^{k+1})$, and the other components remain unchanged. In Step 3, the index $i^k$ is updated in a cycling manner. 

{
\begin{remark}\label{remark:memory}
Similar to other known incremental methods, e.g., IAG, SAG, SAGA, Finito/MISO, the proposed DIAG method requires a memory of order $\mathcal{O}(np)$ which might not be affordable in some large-scale optimization problems. This issue can be resolved by grouping the functions and creating new sets of functions where each one is the average of a subset of functions. If we combine $m$ functions and use the average of them as the new function, the number of active functions reduces to $n/m$ and the required memory decreases to $\mathcal{O}(np/m)$. On the other hand, this process increases the computational complexity of each iteration from one gradient computation to calculation of $m$ gradients. Indeed, there is a trade-off between the memory and computational complexity per iteration which can be optimized based on the application of interest.
\end{remark}
}

\section{Convergence Analysis}\label{sec_Convergence}

In this section, we study the convergence properties of the proposed double incremental aggregated gradient method. 

The following lemma characterizes an upper bound for the error $\|\bbx^{k+1}-\bbx^*\|$ in terms of the errors of the last $n$ iterations.

\begin{lemma}\label{fundamental_lemma}
Consider the proposed double incremental aggregated gradient (DIAG) method in \eqref{incremental_qn_update}. 
If the conditions in Assumption \ref{cnvx_lip} hold, and the stepsize $\eps$ is chosen as $\eps=2/(\mu+L)$, the sequence of iterates $\bbx^k$ generated by DIAG satisfies the inequality
\begin{equation} \label{fundamental_lemma_claim}
\|\bbx^{k+1}-\bbx^*\|\leq  \left( \frac{\kappa-1}{\kappa+1}\right) \left[ \frac{\| \bbx^k-\bbx^*\|+\dots +\| \bbx^{k-n+1}-\bbx^*\|}{n}\right],
\end{equation}
where $\kappa=L/\mu$ is the objective function condition number.
\end{lemma}

\begin{proof}
See Appendix \ref{apx_fundamental_lemma}.
\end{proof}

The result in Lemma \ref{fundamental_lemma} has a significant role in the analysis of DIAG. It shows that the error at step $k+1$ is smaller than the average of the last $n$ errors where the decrement factor is the ratio $(\kappa-1)/(\kappa+1)$ which is strictly smaller than $1$. The cyclic scheme is critical in proving the result in \eqref{fundamental_lemma_claim}, since it allows to replace the sum $\sum_{i=1}^n {\| \bby_i^k-\bbx^*\|}$ by the sum of the last $n$ steps errors $\| \bbx^k-\bbx^*\|+\dots +\| \bbx^{k-n}-\bbx^*\|$. Note that If we pick functions uniformly at random, as in MISO, it is not possible to write the expression in \eqref{fundamental_lemma_claim}, even in expectation. We also cannot write an inequality similar to the one in \eqref{fundamental_lemma_claim} for the IAG method, although it uses a cyclic scheme. This contrast is originated by the difference that IAG only uses gradients average, whereas DIAG uses both variables and gradients averages. In the following theorem, we use the result in Lemma \ref{fundamental_lemma} to show that the sequence of variables $\bbx^k$ converges to the optimal argument $\bbx^*$.

\begin{proposition}\label{CIAG_better_than_GD_lemma}
Consider the proposed double incremental aggregated gradient (DIAG) method in \eqref{incremental_qn_update}, and recall the definition $\rho:=(\kappa-1)/(\kappa+1)$ where $\kappa=L/\mu$ is the problem condition number. If the conditions in Assumption \ref{cnvx_lip} hold, and the stepsize $\eps$ is chosen as $\eps=2/(\mu+L)$, then the residual $\|\bbx^{k}-\bbx^*\|$ of DIAG for iterations $k=1,\dots,n$ satisfies the inequality
\begin{align}\label{eq_lin_result_1}
\|\bbx^{k}-\bbx^*\| \leq  \rho \left[1-\frac{(k-1)(1-\rho)}{n}\right]\|\bbx^{0}-\bbx^*\| ,\quad \for \ k=1,\dots,n,
\end{align}
and for the steps $k>n$ we have
\begin{align}\label{eq_lin_result_2}
\|\bbx^{k}-\bbx^*\| \leq  \rho^{\lfloor \frac{k-1}{n}\rfloor +1} \left[1-\frac{(1-\rho)}{n} \times \min\left\{1,\frac{n-1}{2}\right\}\right]\|\bbx^{0}-\bbx^*\| ,\quad \for \ k>n,
\end{align}
%
%
%
where $\lfloor a \rfloor$ indicates the floor of $a$. \end{proposition}

\begin{proof}
See Appendix \ref{apx_CIAG_better_than_GD_lemma}.
\end{proof}

The first outcome of the result in Proposition \ref{CIAG_better_than_GD_lemma} is the convergence of the sequence $\|\bbx^{k}-\bbx^*\|$ to zero as $k$ approaches infinity. The second result which we formalize in the following corollary shows that the sequence of error converges linearly after each pass over the dataset.

\begin{corollary}\label{imp_cor}
If the conditions in Proposition \ref{CIAG_better_than_GD_lemma} are satisfied, the error of the proposed DIAG method after $m>1$ passes over the set of functions $f_i$, which requires $mn$ gradient evaluations and corresponds to the iterate $k=n(m-1)+1$, is bounded above by 
\begin{equation} \label{pass_IAG}
\|\bbx^{n(m-1)+1}-\bbx^*\|  \leq \rho^{m} 
 \left[1-\frac{(1-\rho)}{n} \times \min\left\{1,\frac{n-1}{2}\right\}\right]\|\bbx^{0}-\bbx^*\|.
\end{equation}
\end{corollary}

\begin{proof}
Since we count the initial $n$ gradient computations, the iterate $\bbx^1$ requires $n$ gradient computations. After the first iteration each step requires only one gradient computation. Therefore, the total number of gradient computations to evaluate $\bbx^k$ is $n+k-1$. Conversely, the variable that has exactly used $mn$ gradients to be evaluated is $\bbx^{n(m-1)+1}$. Thus, by setting $k=n(m-1)+1$ in \eqref{eq_lin_result_2} we obtain the residual of DIAG after $mn$ passes over the set of functions and the claim in \eqref{pass_IAG} follows.
\end{proof}

The result in Corollary \ref{imp_cor} shows that the subsequence of the last iterates of each pass is linearly convergent. Moreover, the result in Corollary \ref{imp_cor} verifies the advantage of DIAG method versus the full gradient descent (GD) method. In particular, it shows that the error of DIAG after $m>1$ passes over the set of functions $f_i$ corresponding to the iterate $k=n(m-1)+1$ is bounded above by $\rho^{m} [1-({(1-\rho)}/{n}) \times \min\{1,({(n-1)}/{2})\}]\|\bbx^{0}-\bbx^*\|$ which is strictly smaller than the upper bound for the error of GD after $m$ iterations ($nm$ gradient computations)  given by $\rho^m\|\bbx^{0}-\bbx^*\|$. Therefore, the DIAG method outperforms GD for any choice of $\kappa$ and $n>1$; DIAG and GD are identical for $n=1$.

 Notice that after the first pass over the set of functions -- iteration $k=1$ for the DIAG method -- the error of DIAG is upper bounded by $\rho\|\bbx^{0}-\bbx^*\|$ based on the result in \eqref{eq_lin_result_1}. This bound is identical to the result for GD after one pass over the set of functions, since the first iterations of GD and DIAG are identical. 



Although the result in Corollary \ref{imp_cor} implies that the DIAG method is preferable with respect to GD and shows linear convergence of a subsequence of iterates, it is not sufficient to prove linear convergence of the whole sequence of iterates generated by DIAG. To be more precise, the result in Corollary \ref{imp_cor} shows that the subsequence of errors $\{\|\bbx^{kn}-\bbx^*\|\}_{k=0}^\infty  $, which are associated with the variables at the end of each pass over the set of functions, is linearly convergent. However, we aim to show that the whole sequence $\{\|\bbx^{k}-\bbx^*\|\}_{k=0}^\infty$ is linearly convergent. To be more precise, our goal is to prove that the sequence of DIAG iterates satisfies $\|\bbx^k-\bbx^*\|\leq a \gamma^k \|\bbx^0-\bbx^*\|$ for a constant $a> 0$ and a positive coefficient $0\leq\gamma<1$. In the following theorem, we show that this condition is satisfied for the DIAG method.

\begin{theorem}\label{thm_lin}
Consider the introduced double incremental aggregated gradient\\ (DIAG) method in \eqref{incremental_qn_update}. 
If the conditions in Assumption \ref{cnvx_lip} hold, and the stepsize $\eps$ is chosen as $\eps=2/(\mu+L)$, for $k\geq 1$ we can write
\begin{equation} \label{lin_IAG}
\|\bbx^{k}-\bbx^*\|  \leq a \gamma^k \|\bbx^{0}-\bbx^*\|,
\end{equation}
if the constants $a>0$ and $0\leq \gamma<1 $ satisfy the following conditions 
\begin{align}
&\rho \left(1-\frac{(k-1)(1-\rho)}{n}\right)\leq a \gamma^k\quad \for\quad  k=1,\dots, n, \label{cond1}\\
& \gamma^{n+1}-\left(1+\frac{\rho}{n}\right)\gamma^n +\frac{\rho}{n}\leq 0\quad \for \quad k>n.\label{cond2}
\end{align}
\end{theorem}

\begin{proof}
See Appendix \ref{apx_linear_convg_lemma}.
\end{proof}

The result in Theorem \ref{thm_lin} provides conditions on the constants $a$ and $\gamma$ such that the linear convergence inequality $\|\bbx^{k}-\bbx^*\|  \leq a \gamma^k \|\bbx^{0}-\bbx^*\|$ holds. However, it does not gurantee that the set of constants $\{a,\gamma\}$ that satisfy the required conditions in \eqref{cond1} and \eqref{cond2} is non-empty. In the following proposition we show that there exist constants $a$ and $\gamma$ satisfying these conditions. 
%
\begin{proposition}\label{prop_non_empty_set}
There exist constants $a>0$ and $0<\gamma<1$ that satisfy the inequalities in \eqref{cond1} and \eqref{cond2}. In other words, the set of feasible solutions for the system of inequalities in \eqref{cond1} and \eqref{cond2} is non-empty.
\end{proposition}

\begin{proof}
See Appendix \ref{apx_prop_non_empty_set}.
\end{proof}

The result in Proposition \ref{prop_non_empty_set} in conjunction with the result in Theorem \ref{thm_lin} guarantees linear convergence of the iterates generated by the DIAG method. 
Although there are different pairs of $\{a,\gamma\}$ that satisfy the conditions in \eqref{cond1} and \eqref{cond2} and lead to the linear convergence result in \eqref{lin_IAG}, we are interested in finding the pair $\{a,\gamma\}$ that leads to the smallest linear convergence factor $\gamma$, i.e., the pair that guarantees the fastest linear convergence rate.
To find the smallest $\gamma$, we should pick the smallest $\gamma$ that satisfies the inequality $ \gamma^{n+1}-\left(1+{\rho}/{n}\right)\gamma^n +{\rho}/{n}\leq 0 $. Then choose the smallest constant $a$ that satisfies the conditions in \eqref{cond1} for the given $\gamma$. To do so, we first look at the properties of the function $h(\gamma):= \gamma^{n+1}-\left(1+{\rho}/{n}\right)\gamma^n +{\rho}/{n}$ in the following lemma.

\begin{lemma}\label{h_function_lemma}
Consider the function $h(\gamma):= \gamma^{n+1}-\left(1+{\rho}/{n}\right)\gamma^n +{\rho}/{n}$ for $\gamma\in[0,1)$. The function $h$ has only one root $\gamma_0$ in the interval $[0,1)$. Moreover, $\gamma_0$ is the smallest choice of $\gamma$ that satisfies the condition in \eqref{cond2} .
\end{lemma}
\begin{proof}
The derivative of the function $h$ is given by 
\begin{equation}
\frac{d}{d \gamma}h =  (n+1)\gamma^n-(n+\rho)\gamma^{n-1}.
\end{equation}
Therefore, the only critical point of the function $h$ in the interval $(0,1)$ is $\gamma^*=(n+\rho)/(n+1)$. The point $\gamma^*$ is a local minimum for the function $h$, since the second derivative of the function $h$ is positive at $\gamma^*$. Notice that the objective function value $h(\gamma^*)<0$ is negative. Moreover, we know that $h(0)>0$ and $h(1)=0$. This observation shows that the function $h$ has a root $\gamma_0$ between $0$ and $\gamma^*$ and this is the only root of function $h$ in the interval $(0,1)$. Thus, $\gamma_0$ is the smallest value of $\gamma$ in the interval $(0,1)$ that satisfies the condition in \eqref{cond2}.
\end{proof}

The result in Lemma \ref{h_function_lemma} shows that the unique root of the function $h(\gamma):= \gamma^{n+1}-\left(1+{\rho}/{n}\right)\gamma^n +{\rho}/{n}$ in the interval $[0,1)$ is the smallest $\gamma$ that satisfies the condition in \eqref{cond2}. We use this result to formalize the pair $\{a,\gamma\}$ with the smallest choice of $\gamma$ which satisfies the conditions in \eqref{cond1} and \eqref{cond2}.

\begin{theorem}\label{thm_best_lin}
Consider the DIAG method in \eqref{incremental_qn_update}. Let the conditions in Assumption \ref{cnvx_lip} hold, and set the stepsize as $\eps=2/(\mu+L)$. Then, the sequence of iterates generated by DIAG is linearly convergent as
\begin{equation} \label{lin_IAG_2}
\|\bbx^{k}-\bbx^*\|  \leq a_0 \gamma_0^k \|\bbx^{0}-\bbx^*\|,
\end{equation}
where $\gamma_0$ is the unique root of the equation 
\begin{align}\label{def_gamma_zero}
\gamma^{n+1}-\left(1+\frac{\rho}{n}\right)\gamma^n +\frac{\rho}{n}=0,
\end{align}
in the interval $[0,1)$ and $a_0$ is given by 
\begin{align}
a_0= \max_{i\in\{1,\dots,n\}}  \rho \left(1-\frac{(i-1)(1-\rho)}{n}\right)\gamma_0^{-i}.
\end{align}
\end{theorem}

\begin{proof}
It follows from the results in Theorem \ref{thm_lin} and Lemma \ref{h_function_lemma}.
\end{proof}

The result in Theorem \ref{thm_best_lin} shows R-linear convergence of the DIAG iterates with the linear convergence factor $\gamma_0$; {however, it does not show that $\gamma_0^n$ is smaller than the linear convergence factor of GD. In the following section, we aim to show that the linear convergence factor of DIAG after $n$ iterations, which is $\gamma_0^n$, is strictly smaller than the linear factor of GD.}

%
%
%


%
\section{Worst-case asymptotic rate of DIAG}\label{sec:worst_case}

{In the previous section, we proved that the DIAG method outperforms GD \textit{after each pass} (Corollary 3), but this result does not characterize the linear convergence factor for the sequence of errors $\|\bbx^k-\bbx^*\|$ generated by DIAG. The result in Theorem 7 shows R-Linear convergence of the sequence  $\|\bbx^k-\bbx^*\|$ to zero; however, it does not show that $\gamma_0^n$ is smaller than the linear convergence factor of GD. In this section, we aim to derive a result that shows the sequence  $\|\bbx^k-\bbx^*\|$ has a linear convergence rate with constant $\gamma_0$ such that $\gamma_0^n$ is strictly smaller than $\rho$, which is the linear convergence factor of GD. To do so, we define the sequence $d^k$ as}
\begin{equation}\label{def_upp_bound}
d^{k+1} = \rho\ \! \frac{d^k + d^{k-1} + \cdots + d^{k-n+1}}{n}
\end{equation}
where $\rho =({\kappa-1})/(\kappa+1)$ and $d^j := \|\bbx^j - \bbx^*\|$ for $j=0,1,2,\dots,n-1$.
 It follows directly from \eqref{fundamental_lemma_claim} that the sequence $d^k$ provides an upper bound for the sequence of the errors $\|\bbx^k - \bbx^*\|$ for all $k\geq0$. In other words, we have the relation
\begin{equation}
   \| \bbx^{k} - \bbx^* \| \leq d^{k} \quad \forall \quad k\geq 0.
   \end{equation}
Next, we characterize the convergence properties of the sequence of $d^k$ which provides an upper bound for the desired error sequence $\|\bbx^k-\bbx^*\|$. To do so, we rewrite the update of the sequence $d^k$ in a matrix form which is more suitable for the analysis. Define the column vector $\bbd^k=[d^{k-1};\dots;d^{k-n}]\in \reals^{n}$ as the concatenation of the last $n$ values of the sequence $d^k$ up to step $k$. Considering the definition of $\bbd^k$ and the update of the sequence $d^k$ we can write     
\begin{equation}
	 \bbd^{k+1} = \bbM_\rho \bbd^k \quad \mbox{where} \qquad 
	 \bbM_\rho :=   \begin{bmatrix}
   						   \frac{\rho}{n} & \frac{\rho}{n} & \hdots & \frac{\rho}{n}       \\
    					                      1 &                   0 & \hdots & 0  \\
    					                      0 &                   1 & \hdots & 0  \\
    					                      0 &                   0 &         1 & 0 
   					  \end{bmatrix}.
 \label{fixed-point-iter}
\end{equation}

We observe that the matrix $\bbM_\rho\in \reals^{n\times n}$ is a non-negative matrix whose eigenvalues determine the asymptotic growth rate of the sequence $\bbd^k$ and hence of $d^k$. It is straightforward to check that the characteristic polynomial of $\bbM_\rho$ is

\begin{equation} T(\lambda) = \lambda^n - \frac{\rho}{n} \lambda^{n-1} - \frac{\rho}{n} \lambda^{n-2} - \hdots -  \frac{\rho}{n} = \frac{\lambda^{n+1}-\left(1+\frac{\rho}{n}\right)\lambda^n +\frac{\rho}{n}}{\lambda - 1} 
\label{char_polynomial}
\end{equation}
whose roots are the eigenvalues of the matrix $\bbM_\rho$. In the remainder of this section, we will infer information about the eigenvalues of $\bbM_\rho$ using Perron-Frobenius (PF) theory. This theory is well developed for positive matrices where all the entries are strictly positive but $\bbM_\rho$ has zero entries and is therefore not positive. Nevertheless, the PF theory has been successfully extended to certain non-negative matrices called \emph{irreducible} matrices. %
A square matrix $\bbA$ is called irreducible if for every $i$ and $j$, there exists an $r$ such that $\bbA^r(i,j)>0$. In the next lemma, we prove that the matrix $\bbM_\rho$ is irreducible which will justify our use of PF theory developed for irreducible matrices. 

\begin{lemma} The matrix $\bbM_\rho$ is irreducible for any $\rho>0$.  
\end{lemma}

\begin{proof} By the definition of irreducibility, we need to show that for every $i$ and $j$, there exists an $r$ such that the $r$-th power of the matrix $\bbM_\rho$ is entrywise positive, i.e. $\bbM_\rho^r(i,j)>0$. Let $\bbe_1, \bbe_2, \dots, \bbe_n$ be the standard basis for $\mathbb{R}^n$. It suffices to show that we can choose $r=n$ for all $i$ and $j$, i.e. 
    \begin{equation} \bbM_\rho^n(i,j) = \bbe_i^T \bbM_\rho^n \bbe_j > 0, \quad \forall \quad i, j.
    \label{suff-cond-irreducible}
    \end{equation}
By the definition of the recurrence \eqref{fixed-point-iter}, we have 
	$$ \begin{bmatrix}
    d^{2n-1}        \\
    d^{2n-2} \\
    \hdots \\
    d^{n}
\end{bmatrix} = \bbd^{2n} = (\bbM_\rho)^n \bbd^n =  (\bbM_\rho)^n \begin{bmatrix}
    d^{n-1}        \\
    d^{n-2} \\
    \hdots \\
    d^{0}
\end{bmatrix}.$$
Fix any $j \in \{1, \dots, n\}$ and choose $\bbd^n = \bbe_j$ (this would correspond to the initalization $d_{n-\ell} = 1$ for $\ell = j$ and $d_{n-\ell} = 0$ for $1\leq \ell\leq n$ and $\ell \neq j$). Then, 
%
using the definition of $\bbM_\sigma$, it is easy to check that such an initialization of $\bbd^n$ leads to $d^{n} = \rho/n > 0$, $d^{n+1}>0, \dots, d^{2n-1}>0$. Therefore, for every $i$ and $j$, we have 
\begin{equation}
(\bbM_\rho)^n(i,j) = \bbe_i^T (\bbM_\rho)^n \bbe_j  = d^{2n-i}> 0
\end{equation}
which proves \eqref{suff-cond-irreducible} and completes the proof.
\end{proof}

The following result shows that  the sequence $d^k$ converges to zero linearly with a constant $\gamma_0$  where $\gamma_0$  is defined by \eqref{def_gamma_zero}. We also derive upper and lower bounds on $\gamma_0$.

\begin{theorem}\label{theorem-rate-bounds} Consider the constant $\rho=(\kappa-1)/(\kappa+1) \in (0,1)$ and let $\lambda^*(\rho)$ be the spectral radius of the matrix $\bbM_\rho$. Then, 

\begin{enumerate}
	
        \item [$(i)$] $\lambda^*(\rho)$ is the largest real root of the the polynomial characteristic polynomial $T(\lambda)$. Furthermore, it is a simple root.
		  \item [$(ii)$] We have the limit 
		  \begin{equation}\label{asym_lin_rate}
		  \lim_{k\to\infty} d^{k+1} / d^k = \lambda^*(\rho).
		  \end{equation}
		  \item [$(iii)$] For integer numbers $n>1$ the constant $\lambda^*(\rho)$ is bounded below and above as
		 		\begin{equation}\label{bounds_on_rate}
			\rho	\leq \lambda^*(\rho) <\sqrt[n]{\rho} .
				\end{equation}
		 \item [$(iv)$] We have $\lambda^*(\rho) = \gamma_0$ where $\gamma_0$ is the largest real root of the polynomial $h(\lambda):=\lambda^{n+1}-\left(1+({\rho}/{n})\right)\lambda^n +({\rho}/{n})$ in the interval $[0,1)$.		
\end{enumerate} 
\end{theorem}

\begin{proof} 
See Appendix \ref{apx_theorem-rate-bounds}.
\end{proof}


{The results in Theorem \ref{theorem-rate-bounds} study the convergence properties of the sequence $d^k$ defined in \eqref{def_upp_bound} which is an upper bound for the sequence of DIAG error $\|\bbx^k-\bbx^*\|$. The last result in Theorem \ref{theorem-rate-bounds} shows that the constant $\lambda^*(\rho)$ which is the spectral radius of the matrix $\bbM_\rho$, is equal to the linear convergence factor $\gamma_0$ of DAIG defined as the root of the polynomial in \eqref{def_gamma_zero}. The second result indicates that the sequence $d^k$ has an asymptotic linear convergence rate with the constant $\lambda^*(\rho)=\gamma_0$. This convergence result
 is not stronger than the result in Theorem \ref{thm_best_lin}, since it holds asymptotically, but we report this result since it shows that there exists a sequence which achieves the theoretical upper bound proven for the DIAG method. Also, this result is interesting since it proves the R-linear convergence of DIAG from an entirely different approach based on Perron-Frobenius theory. The most important result in Theorem \ref{theorem-rate-bounds} is the third result which shows that $\lambda^*(\rho)$, which is equal to $\gamma_0$ according to the last result, is strictly smaller than $\sqrt[n]{\rho}$. Based on the inequality in \eqref{bounds_on_rate}, we obtain that the linear convergence factor of DIAG after running for $n$ iterations, which is $\gamma_0^n$, is strictly smaller than $\rho$ the decrement factor of GD after one pass over the set of functions.}

{It is worth mentioning that the upper bound sequence ${d_k}$ achieves the asymptotic Q-linear convergence with ration $\gamma_0$ doesn't necessarily mean the sequence $||\bbx^k-\bbx^*||$ cannot achieve a better rate, but it implies the rate cannot be worse than $\gamma_0$. Therefore, we refer to this result as the asymptotic \textit{worst-case scenario} analysis of DIAG.}

{
\begin{remark}
Note that  the right inequality in \eqref{bounds_on_rate} can also be achieved using the results in Section \ref{sec_Convergence}. To be more specific, first one may show $\rho=(1-\frac{1-\rho}{1})<(1-\frac{1-\rho}{2})^2<\dots<(1-\frac{1-\rho}{n})^n$. This sequence of inequalities implies that $\rho<(1-\frac{1-\rho}{n})^n$, which is equivalent to the inequality $\rho^{\frac{n+1}{n}}-(1+\frac{\rho}{n})\rho +\frac{\rho}{n}<0$, and, therefore, we obtain that $h(\rho^{1/n})<0$. Further, according to the result in Lemma \ref{h_function_lemma}, $\lambda^*(\rho)=\gamma_0$ is the unique solution to $h(\lambda) = 0$, i.e., $h(\lambda^*(\rho))=0$, and $\lambda^*(\rho)<\lambda$ for all $\lambda \in [0, 1)$ with $h(\lambda) < 0$. Combining these two results leads to the conclusion that $\lambda^*(\rho)<\rho^{1/n}$.

\end{remark}
}

\begin{figure}
\centering
\includegraphics[width=0.49\linewidth]{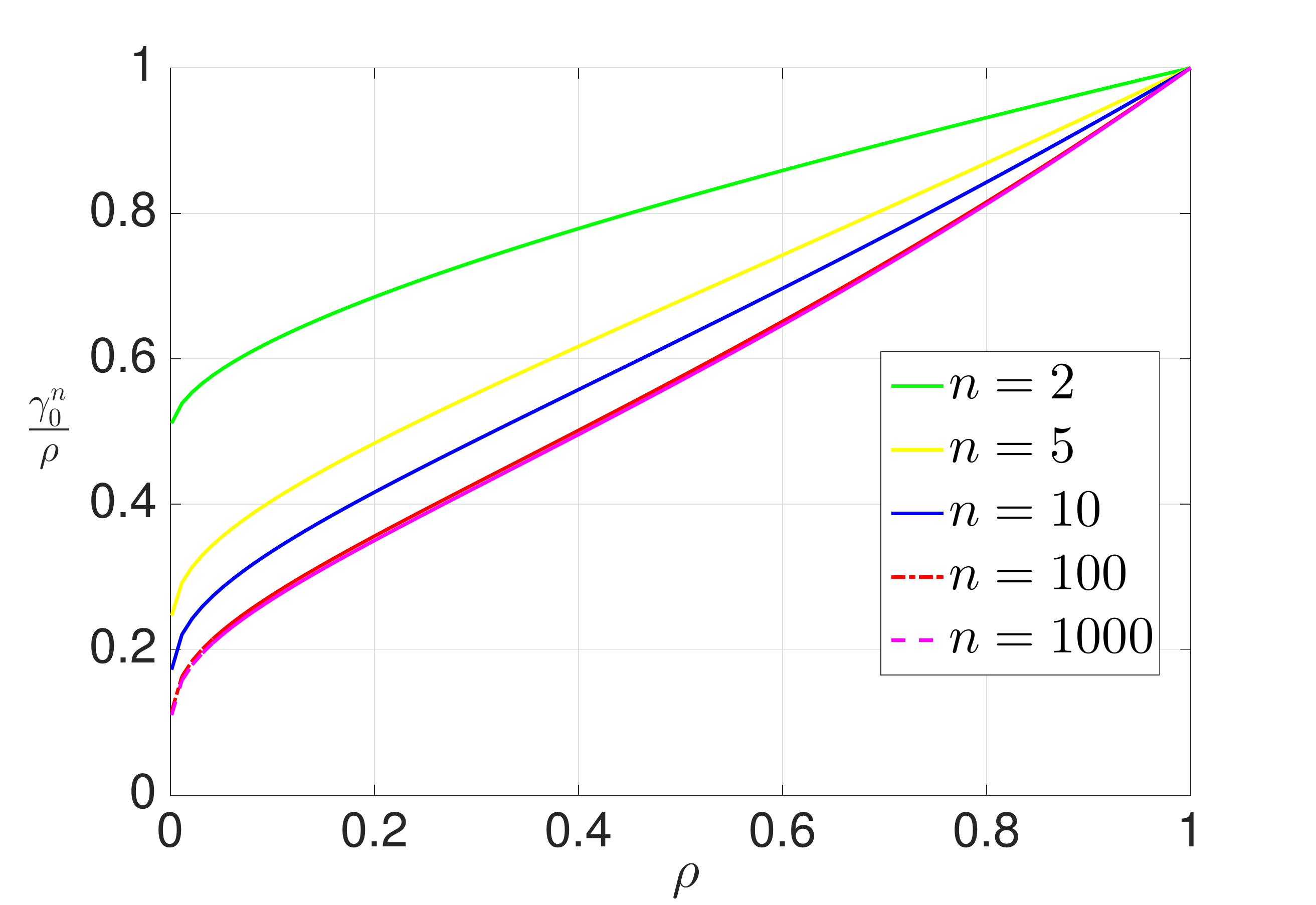}
\includegraphics[width=0.49\linewidth]{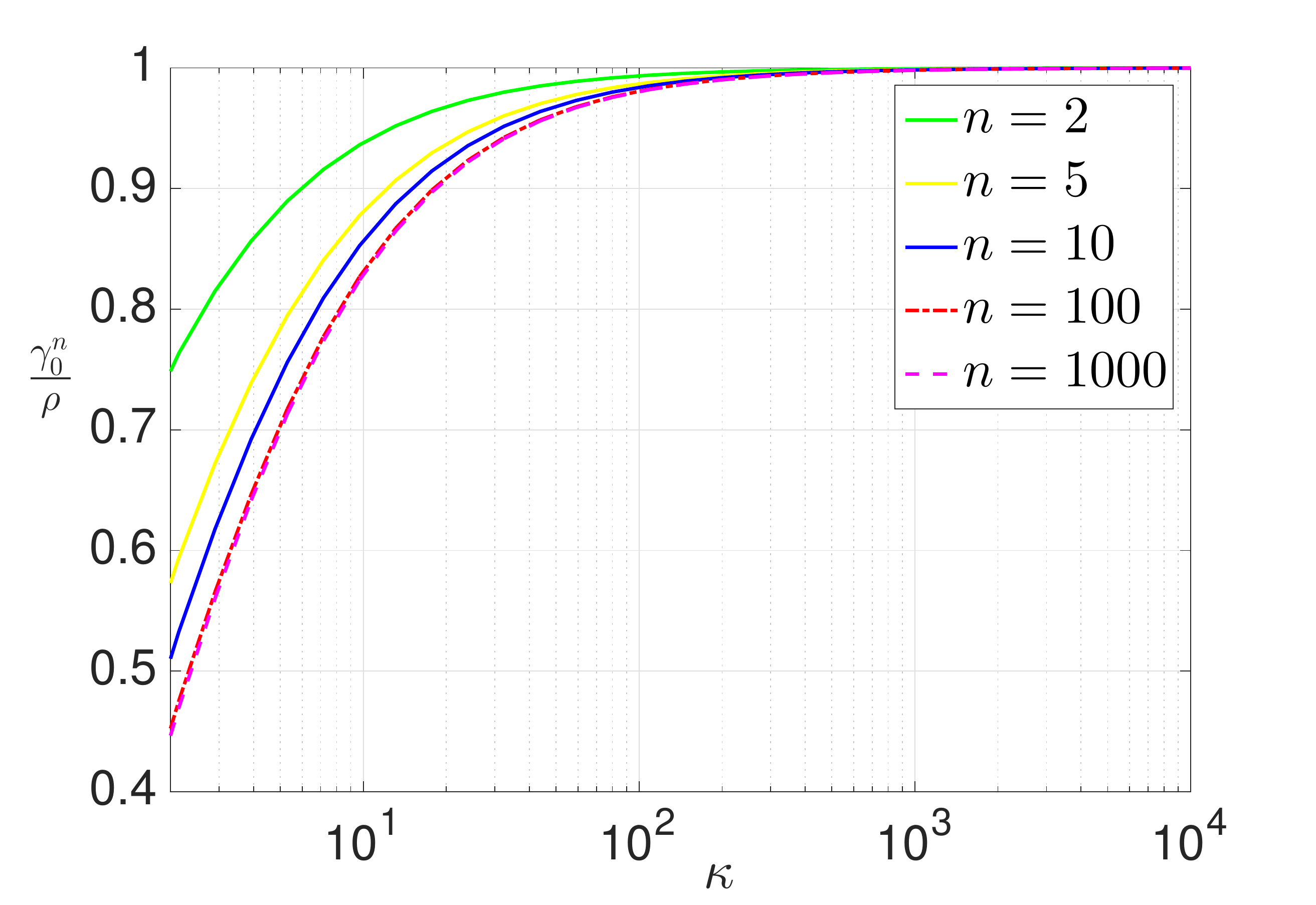}
\vspace{-3mm}
\caption{Comparison of the linear convergence factors of DIAG and GD via the ratio $\gamma_0^n/\rho$ in terms of $\rho$ (left) and $\kappa$ (right) for different choices of $n$.}
\vspace{-3mm}
\label{fig:comp} \end{figure}

{Indeed, formalizing the gap between the linear convergence factors of DIAG and GD requires access to an explicit expression for the largest root of the polynomial in~\eqref{char_polynomial}. However, for specific choices of $n$ and $\kappa$ one can evaluate the DIAG linear convergence factor $\gamma_0$ using polynomial solvers and compare it with the linear factor of GD. We, therefore, compare the ratio $\gamma_0^n/\rho$ for some choices of $\rho$ and $n$ by finding the root of the polynomial using a MATLAB solver. The outcome of the comparison is illustrated in Fig. \ref{fig:comp}.  As we observe in the left plot in Fig. \ref{fig:comp}, for the case of $n=2$, the variable $\gamma_0^n/\rho$, which is the ratio between the linear convergence factors of DIAG and GD after one pass over the functions, is close to $0.5$ for small choices of $\rho$, while it approaches $1$ as $\rho$ becomes closer to $1$. Therefore, for smaller choices of $\rho$ the gain in using DIAG instead of GD is more significant comparing to the cases that $\rho$ is close to $1$. Similar pattern can be observed for other choices of $n$. Conversely, for a fixed choice of $\rho$, when the number of functions $n$ increases the ratio $\gamma_0^n/\rho$ becomes smaller. This behavior shows that by increasing the number of functions $n$ the gap between the performances of DIAG and GD increases and DIAG becomes more favorable. Since $\rho=(k-1)/(k+1)$ is an increasing function of the problem condition number~$\kappa$, similar conclusions can be achieved by comparing the ratio $\gamma_0^n/\rho$ for different choices of $n$ and $\kappa$ as demonstrated in the right plot in Fig.~\ref{fig:comp}. It is also worth mentioning, in all the illustrated curves, the ratio $\gamma_0^n/\rho$ is smaller than $1$ which verifies our theoretical conclusion that DIAG outperforms GD for all choices of $n$ and $\kappa$.}

\section{Numerical experiments}\label{sec:simulations}

In this section, we study the performance of the proposed DIAG method and compare it with existing alternative first-order methods. To do so, we first apply DIAG to solve a family of quadratic programming problems. Then, we evaluate the performance of DIAG and other first-order methods in solving a logistic regression minimization problem.

\subsection{Quadratic programming example}
To study the effect of number of functions $n$ and problem condition number $\kappa$ on the performance of the GD, IAG, and DIAG methods, we first apply these algorithms in solving a quadratic  programming problem, where we can tune the problem condition number. In particular, consider the optimization problem 
\begin{align}\label{eq_simulation_problem}
\min_{\bbx \in \reals^p} f(\bbx) := \frac{1}{n}\sum_{i=1}^n \frac{1}{2} \bbx^T \bbA_i \bbx + \bbb_i^T \bbx,
\end{align}
where each matrix $\bbA_i \in \reals^{p \times p}$ is a positive definite diagonal matrix and each vector $\bbb_i \in \reals^p$ is randomly chosen from the box $[0,1]^p$. To control the problem condition number, the first $p/2$ diagonal elements of $\bbA_i$ are chosen uniformly at random from the interval $[1, 10^1, \hdots, 10^{\eta/2}]$ and its last $p/2$ elements chosen from the interval $[1, 10^{-1}, \hdots, 10^{-\eta/2}]$. This selection resulting in the sum matrix $\sum_{i=1}^n \bbA_i$ having eigenvalues in the range $[n 10^{-\eta/2}, n 10^{\eta/2}]$. In our experiments, we fix the variable dimension as $p=20$ and the number of functions as $n=200$. Moreover, the stepsizes of GD and DIAG are set as their best theoretical stepsizes which are $\eps_{GD}=2/(\mu+L)$ and $\eps_{DIAG}=2/(\mu+L)$, respectively. {Note that the stepsize suggested in \cite{gurbuzbalaban2015convergence} for IAG is $\eps_{IAG}=(0.32\mu)/((nL)(L+\mu))$; however, this choice of stepsize leads to slow convergence of IAG in practice. Thus, we use the stepsize $\eps_{IAG}=2/(nL)$ which performs better than the one suggested in \cite{gurbuzbalaban2015convergence}. }

We compare these methods in terms of the total number of gradient evaluations. Note that comparing these methods in terms of the total number of iterations would not be fair since each iteration of GD requires $n$ gradient evaluations, while IAG and DIAG only require one gradient computation per iteration.


%
\begin{figure} [t]
\centering
\includegraphics[width=0.495\linewidth]{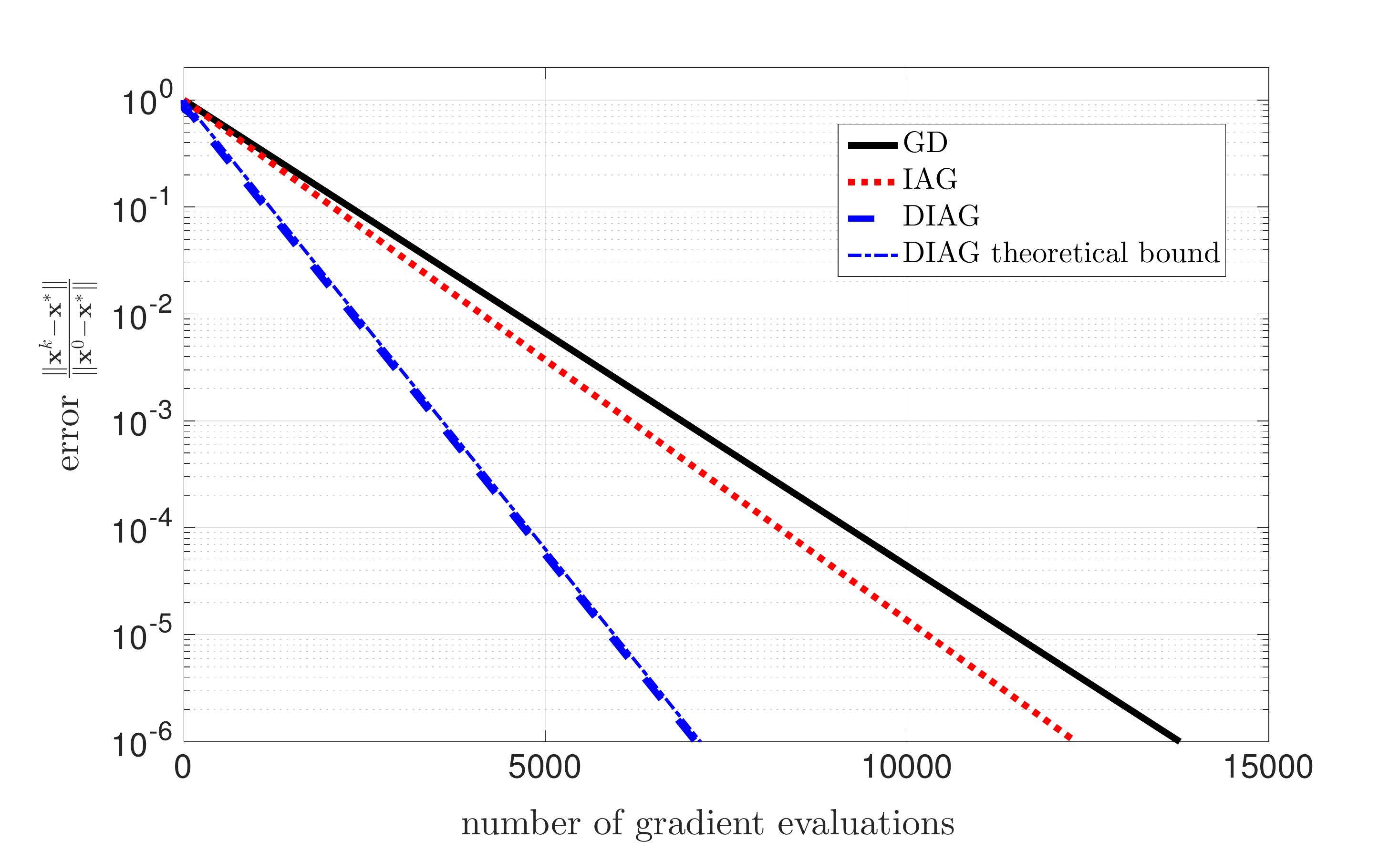}
\includegraphics[width=0.495\linewidth]{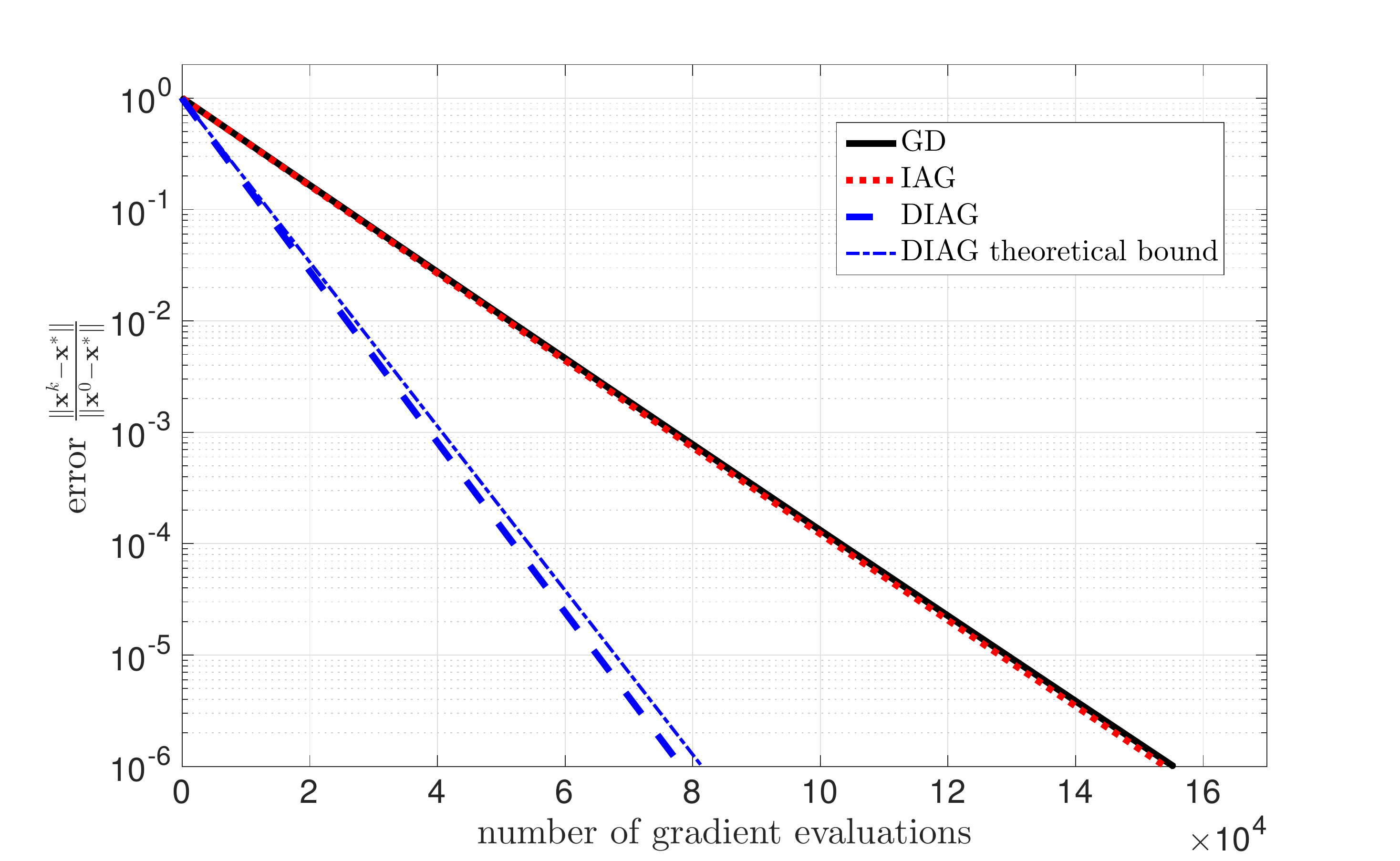}
\vspace{-3mm}
\caption{Relative error of GD, IAG, and DIAG versus number of gradient evaluations for the quadratic programming in \eqref{eq_simulation_problem} with $n=200$ and $\kappa=10$ (left) and $n=200$ and $\kappa=117$ (right). When the condition number is small, the IAG method performs slightly better than GD, while DIAG has the fastest convergence path. For the case that the condition number is larger, IAG and GD have similar convergence paths, and the best performance belongs to DIAG.}
\vspace{-3mm}
\label{fig:1} \end{figure}


We first consider the case that $\eta=1$ and use the realization with condition number $\kappa=10$ to have a relatively small condition number. The left plot in Fig.~\ref{fig:1} demonstrates convergence paths of the normalized error $\|\bbx^k - \bbx^*\|/{\|\bbx^0 - \bbx^*\|}$ for IAG, DIAG, and GD when $n=200$ and $\kappa=10$. As we observe, IAG performs slightly better than GD, while the best performance belongs to DIAG. To be more precise, DIAG requires $7,069$ gradient evaluations (approximately $35$ passes over the dataset) to achieve the relative error of $\|\bbx^k-\bbx^*\|/\|\bbx^0-\bbx^*\|=10^{-6}$, while IAG requires $12,330$ gradient evaluations (approximately $61$ passes over the dataset) to achieve the same accuracy. The GD method has the worst performance and achieves the relative error $\|\bbx^k-\bbx^*\|/\|\bbx^0-\bbx^*\|=10^{-6}$ after $68$ iterations which is equivalent to $13,600$ gradient evaluations. {We have also illustrated the theoretical bound for the DIAG method in Fig. \ref{fig:1}, which is computed by finding the root of the polynomial in \eqref{def_gamma_zero} for  $n=200$ and $\rho=(\kappa-1)/(\kappa+1)=9/11$. In this case, the root of the polynomial is $\gamma_0=0.998067$ and the DIAG theoretical bound curve corresponds to the sequence $0.998067(\|\bbx^k-\bbx^*\|/\|\bbx^0-\bbx^*\|)$. As we observe, the performance of the DIAG method is almost identical to its proven theoretical bound which shows the tightness of the bound for DIAG.
 }

Comparison of the convergence guarantees for GD, IAG, and DIAG shows that the IAG method is more sensitive to the problem condition number, since the linear convergence factor of IAG is of the order $1-\mathcal{O}(1/\kappa^2)$, while the linear convergence factor of GD and DIAG are at the order of  $1-\mathcal{O}(1/\kappa)$. To study the effect of problem condition number in practice, we increase the constant $\eta$ to have a poorly conditioned problem. In particular, we increase the problem condition number by setting $\eta=2$ and using the realization with condition number $\kappa=117$. The right plot in Fig.~\ref{fig:1} illustrates performance of these methods for the case that $n=200$ and $\kappa=117$. We observe that the convergence path of IAG is almost identical to the one for GD. This observation verifies that the performance of IAG worsens more significantly by increasing the problem condition number. Interestingly, the relative performance of DIAG and GD does not change by increasing the problem condition number. To be more specific, for the case that $n=200$ and $\kappa=117$, GD and IAG reach the relative error $\|\bbx^k-\bbx^*\|/\|\bbx^0-\bbx^*\|=10^{-6}$ after $1.54\times 10^5$ gradient evaluations, while DIAG requires only $7.8\times 10^4$ gradient computations to achieve the same accuracy. {As in the previous case, we also compare the performance of DIAG with its proven theoretical bound. To do so, we find the root of the polynomial in \eqref{def_gamma_zero} for  $n=200$ and $\rho=(\kappa-1)/(\kappa+1)=116/118$ which is $\gamma_0=0.99983$. As we observe the convergence path of DIAG is very close to the proven  theoretical upper bound for this method.}

{Although $n$ iterations of DIAG and IAG and one iteration of GD have the same complexity in terms of the total number of gradient evaluations, DIAG and IAG require more elementary operations than GD due to averaging of the gradients. In many large scale machine learning applications the bottleneck is the computation of gradients; however, in the special case of quadratic programming problems, the additional elementary operations that IAG and DIAG require for computing the averages cannot be neglected. Therefore, to have a fair comparison between the incremental methods, i.e., IAG and DIAG, and the full-batch method, i.e., GD, we also compare these methods in terms of runtime as shown in Fig.~\ref{fig:time_quad} for the quadratic programming problem given by \eqref{eq_simulation_problem}. We observe in Fig.~\ref{fig:time_quad} that the performance of GD becomes better relative to the incremental methods. In particular, for the case of $n=200$ and $\kappa=10$, we observe that GD outperforms IAG and is marginally worse than DIAG. For the case of $n=200$ and $\kappa=117$, where $n$ is not significantly larger than $\kappa$, we observe that GD performs significantly better than IAG, while the convergence paths of GD and DIAG are close to each other. These observations lead to the conclusion that for quadratic programming problems, where the gradient evaluations are simply elementary operations, the cost of computing the averages in IAG and DIAG cannot be neglected.}

%
\begin{figure} [t]
\centering
\includegraphics[width=0.495\linewidth]{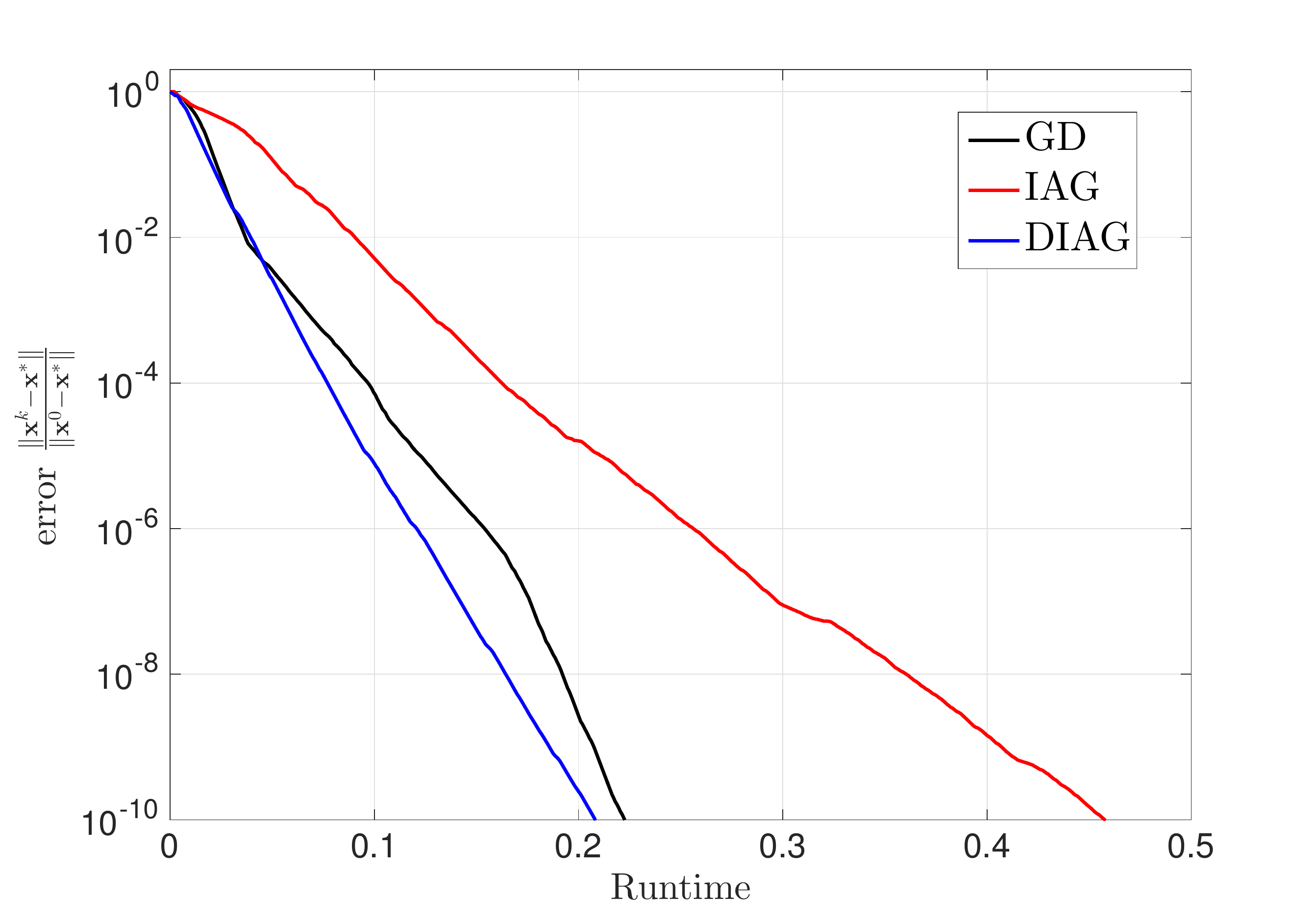}
\includegraphics[width=0.495\linewidth]{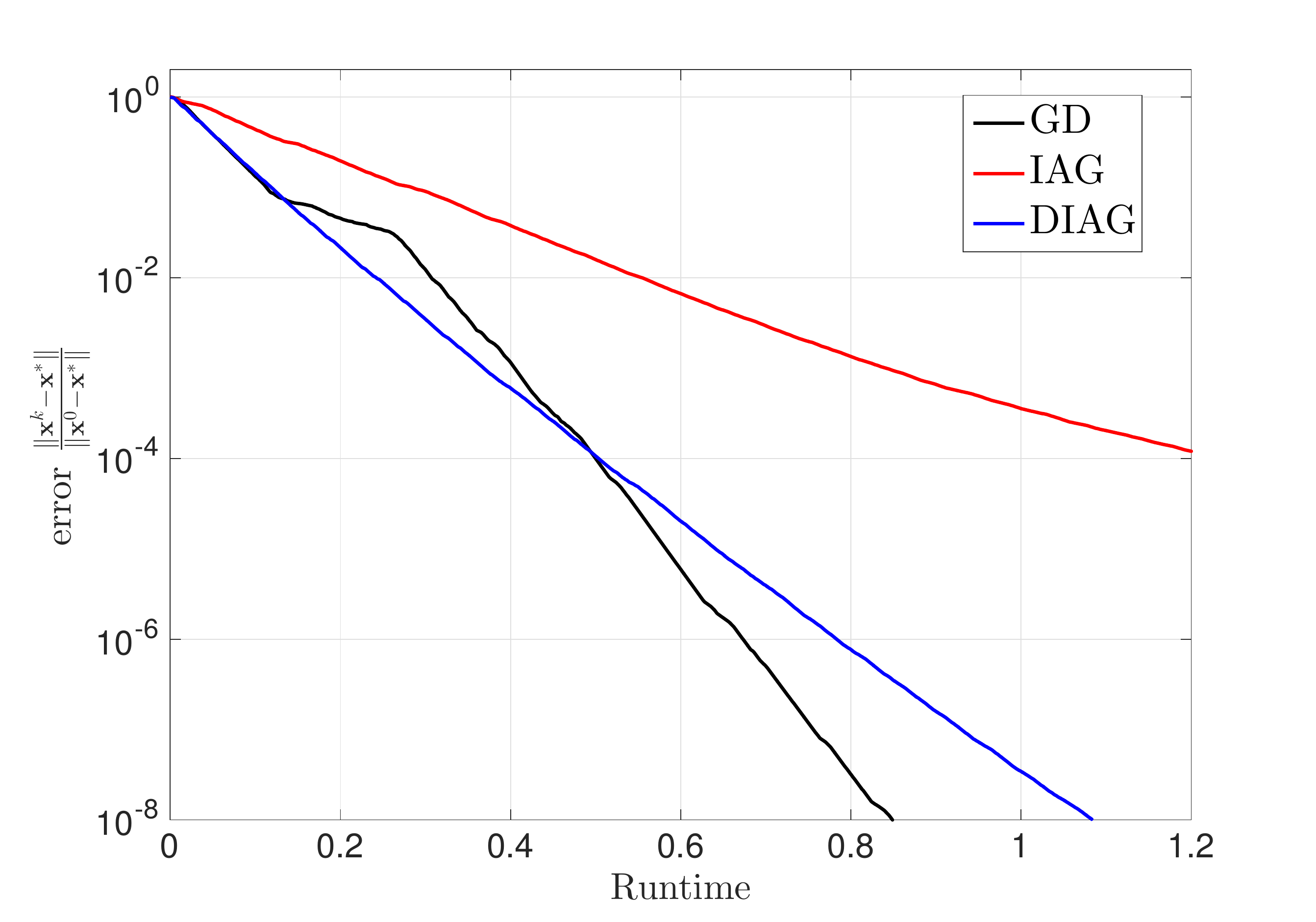}
\vspace{-3mm}
\caption{Relative error of GD, IAG, SAG, Finito, and DIAG versus number of gradient evaluations for the quadratic programming in \eqref{eq_simulation_problem} with $n=200$ and $\kappa=10$ (left) and $n=200$ and $\kappa=117$ (right).}
\vspace{-3mm}
\label{fig:time_quad} \end{figure}

%
\begin{figure} [t]
\centering
\includegraphics[width=0.495\linewidth]{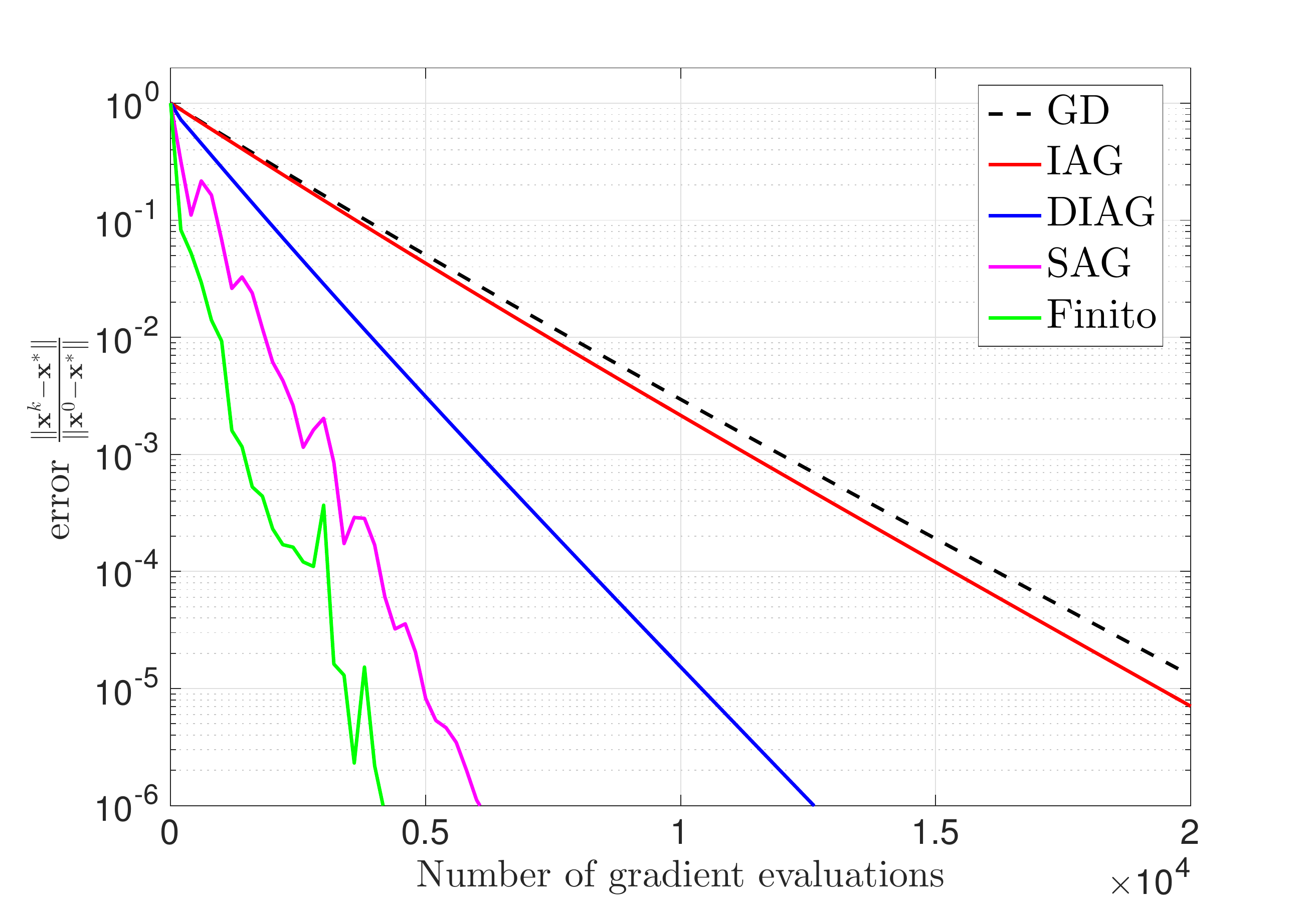}
\includegraphics[width=0.495\linewidth]{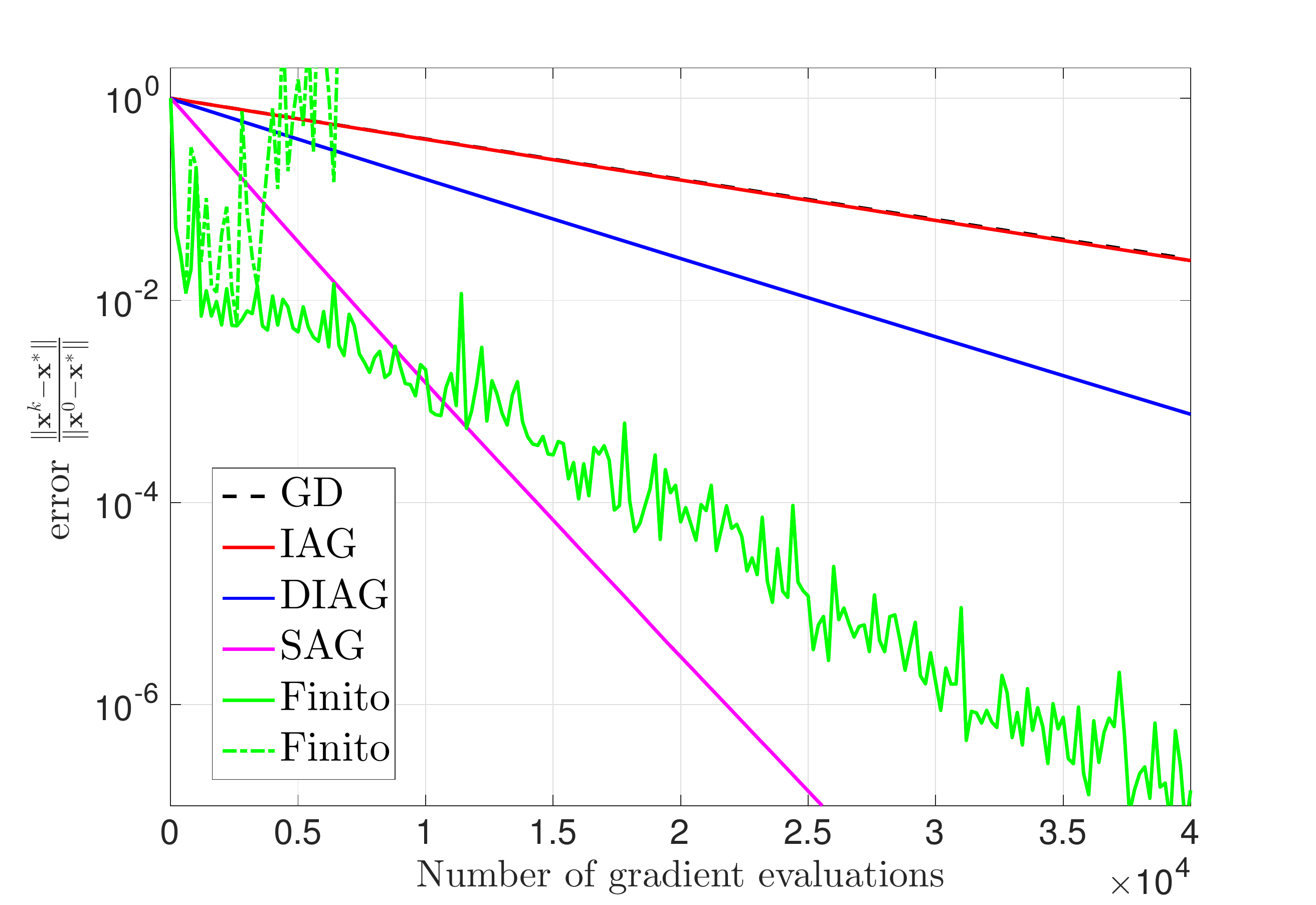}
\vspace{-3mm}
\caption{Relative error of GD, IAG, SAG, Finito, and DIAG versus number of gradient evaluations for the quadratic programming in \eqref{eq_simulation_problem} with $n=200$ and $\kappa=19$ (left) and $n=200$ and $\kappa=120$ (right).}
\label{fig:221} \end{figure}
%

{We also compare the proposed DIAG method with SAG \cite{schmidt2017minimizing,roux2012stochastic} and Finito/MISO \cite{defazio2014finito,mairal2015incremental} which are among the most successful stochastic incremental methods for solving the finite sum minimization problem in \eqref{org_prob}. For SAG we use the stepsize $1/16L$ as suggested in \cite{schmidt2017minimizing}, and for Finito algorithm we use the stepsize $1/2\mu$ as suggested in \cite{defazio2014finito}. The left plot in Fig.~\ref{fig:221}, which corresponds to the case that $n=200$ and $\kappa=19$, shows that the performance of SAG and Finito are better than the one for DIAG, while they fluctuate more comparing to IAG and DIAG. The gap between the performances of IAG and DIAG, and their stochastic variants SAG and Finito comes from the fact that convergence guarantees of SAG and Finito hold for larger choices of stepsize comparing to the ones for IAG and DIAG. But this improvement comes at the cost of moving from a deterministic convergence guarantee (for IAG and DIAG) to results that hold in expectation (for SAG and Finito) which might lead to volatile convergence paths as shown in the left plot in Fig.~\ref{fig:221}. 

The right plot in Fig.~\ref{fig:221} illustrates the convergence paths of GD, IAG, SAG, Finito, and DIAG for a problem with large condition number $\kappa=120$. We observe that SAG and Finito outperform DIAG; however, their convergence guarantees hold in expectation which is a much weaker notion of convergence compared to deterministic convergence guarantees. As an example, one realization of Finito in the right plot in Fig.~\ref{fig:221} performs pretty well, while the other one diverges. In contrast, the results for IAG and DIAG are deterministic, and for any realization of the iterates convergence to the optimal solution is guaranteed.}

\subsection{Logistic Regression minimization}
 
In this section, we compare the performance of GD, IAG, and DIAG in solving a binary classification problem. Consider the given training set $\mathcal{S}=\{\bbu_{i},l_i\}_{i=1}^{i=n}$ which contains $n$ realizations of the feature vectors $\bbu_i \in \reals^p$ and respective label $l_i$ where the labels are either $-1$ or $1$. The goal is to find the optimal classifier $\bbx^* \in \reals^p$ that minimizes the regularized logistic loss which is given by
\begin{equation}\label{logistic_problem}
\min_{\bbx\in \reals^p} f(\bbx):= \frac{1}{n} \sum_{i=1}^n \log(1+\exp(-l_i\bbx^T\bbu_i ))  + \frac{\lambda}{2}  \|\bbx\|^2,
\end{equation}
where the regularization term $({\lambda}/{2})  \|\bbx\|^2$ is added to avoid overfitting. The problem in \eqref{logistic_problem} is a particular case of the problem in \eqref{org_prob} when the function $f_i$ is defined as $f_i(\bbx)= \log(1+\exp(-l_i\bbx^T\bbu_i ))  +({\lambda}/{2})  \|\bbx\|^2$.

Note that the objective function $f$ in \eqref{logistic_problem} is strongly convex with the constant $\mu=\lambda $ and its gradients are Lipschitz continuous with the constant $L=\lambda+\zeta/4$ where $\zeta=\max_{i} \bbu_i^T\bbu_i$. It is easy to verify that the instantaneous functions $f_i$ are also strongly convex with constant $\mu=\lambda$, and their gradients are Lipschitz continuous with constant $L=\lambda+\zeta/4$. This observation shows that condition in Assumption \ref{cnvx_lip} hold for the logistic regression problem in \eqref{logistic_problem}. {In this experiment, we normalize the samples to set the parameter $\zeta=1$. Further, the regularization parameter is chosen as $\lambda=1/\sqrt{n}$.}

We apply GD, IAG, and DIAG to solve the logistic regression problem in \eqref{logistic_problem} for the MNIST dataset \cite{lecun1998mnist}. We only use the samples that correspond to digits $0$ and $8$ and assign label $l_i=1$ to the samples that correspond to digit $8$ and label $l_i=-1$ to those associated with digit $0$. We get a total of $n=11,774$ training examples, each of dimension $p=784$.  
The objective function error $f(\bbx^k)-f(\bbx^*)$ of the GD, IAG, and DIAG methods versus the number of passes over the dataset are shown in the left plot in Fig. \ref{fig:3}. We report the results for the stepsizes $\eps_{GD}=2/(\mu+L)$, $\eps_{IAG}=2/(nL)$, and $\eps_{DIAG}=2/(\mu+L)$ as in the quadratic programming. We observe that the proposed DIAG method outperforms GD and IAG. 

{As we discussed in the quadratic programming example, $n$ iterations of  DIAG or IAG require more elementary operations than a single iteration of GD. Hence, we also compare these methods in terms of runtime as shown in the right plot in Fig. \ref{fig:3}. Note that in this case, in contrast to the quadratic programming example, gradient evaluations are more costly than the elementary operations required in the update and therefore we expect to gain more by running incremental methods. Indeed, we observe in the right plot in Fig.~\ref{fig:3} that the DIAG method outperforms GD significantly in terms of runtime. However, the performance of IAG and GD are almost similar to the one for GD. Comparing these results with the quadratic programming example shows that in terms of runtime incremental methods are more preferable in cases that gradient evaluation is more costly than elementary operations.}

%
\begin{figure} [t]
\centering
\includegraphics[width=0.49\linewidth]{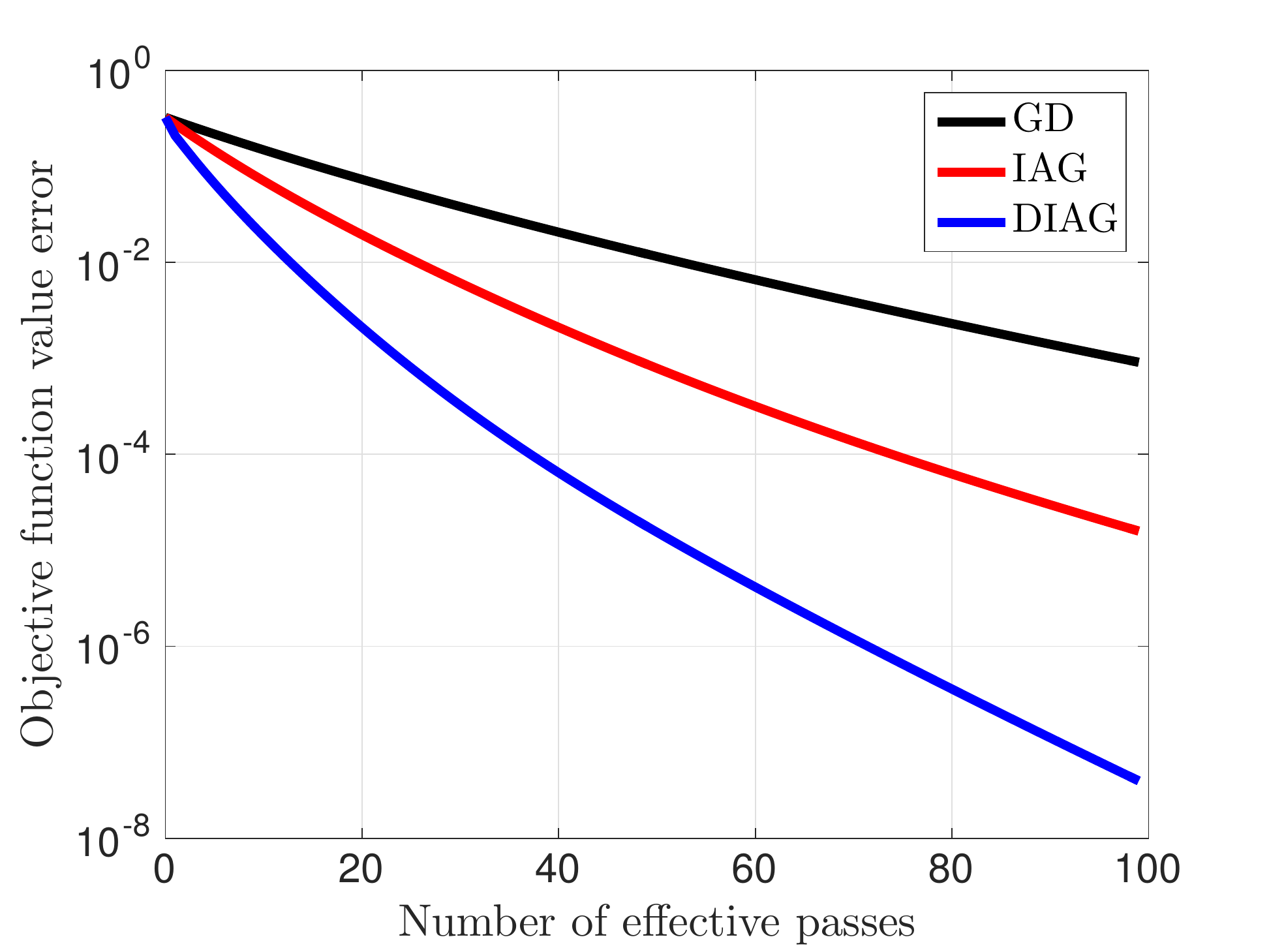}
\includegraphics[width=0.49\linewidth]{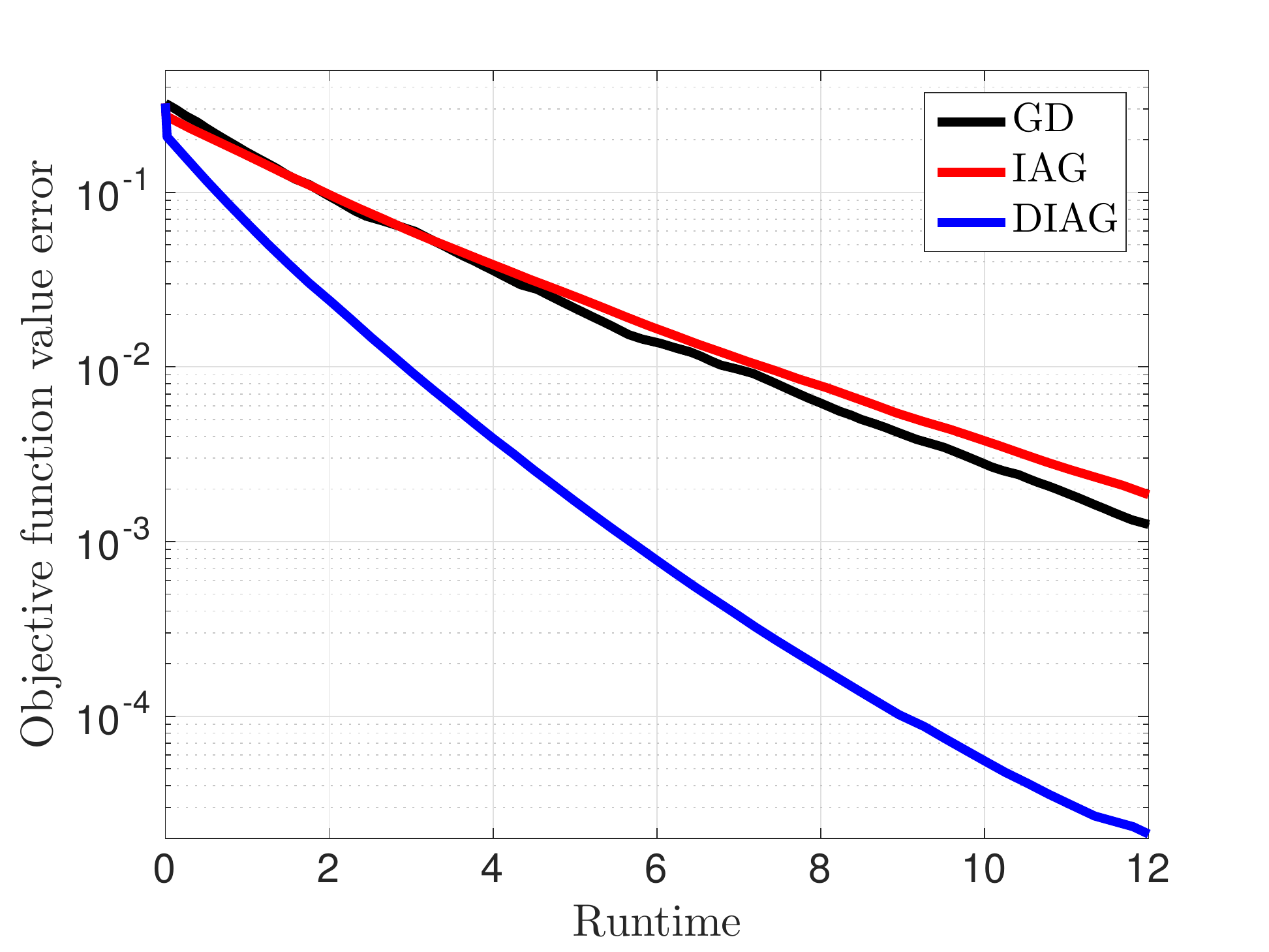}
\caption{Convergence paths of GD, IAG, and DIAG in terms of number of effective passes over the dataset (left) and runtime (right) for the binary classification application.}
\label{fig:3} \end{figure}

\section{Conclusion}\label{sec:conclusions} 
In this paper we proposed a novel incremental method for solving the average of a set of $n$ smooth and strongly convex functions. The proposed double incremental aggregated gradient method (DIAG) uses the aggregated average of both variables and gradients to update its iterate in oppose to classic cyclic incremental methods that only use the average of gradients. The convergence analysis of the DIAG method guarantees improvement with respect to the gradient descent (GD) method. This result makes DIAG the first cyclic incremental method that improves GD under all circumstances. Moreover, we showed that the sequence of iterates generated by DIAG is linearly convergent. Numerical experiments matched the theoretical results and showcased the advantage of DIAG relative to GD and the classic incremental aggregated gradient method (IAG).  

As a future research direction, we aim to extend the double incremental idea to the accelerated gradient descent method (AGD) to obtain an incremental method that surpasses the optimal AGD method under any circumstances.

\appendix

\section{Proof of Lemma \ref{fundamental_lemma}}\label{apx_fundamental_lemma}

Consider the update in \eqref{incremental_qn_update}. Subtract the optimal argument $\bbx^*$ from both sides of the equality to obtain
\begin{align}\label{basic_lemma_proof_100}
\bbx^{k+1}-\bbx^*
&=		 \frac{1}{n} \sum_{i=1}^n(\bby_i^k-\bbx^*) - \frac{\eps}{n} \sum_{i=1}^n \nabla f_i (\bby_i^k) .
\end{align}
Note that the global objective function gradient at the optimal point is null, i.e., $(1/n)\sum_{i=1}^n \nabla f_i (\bbx^*) =\bb0$. This observation in conjunction with the expression in \eqref{basic_lemma_proof_100} leads to 
\begin{align}\label{basic_lemma_proof_200}
\bbx^{k+1}-\bbx^*
&=
\frac{1}{n} \sum_{i=1}^n\left(\bby_i^k-\bbx^*\right) - \frac{\eps}{n} \sum_{i=1}^n \left(\nabla f_i (\bby_i^k)-\nabla f_i (\bbx^*)\right) 
\nonumber\\
 &=
\frac{1}{n} \sum_{i=1}^n\Big[ \bby_i^k-\bbx^* -\eps \left(\nabla f_i (\bby_i^k)-\nabla f_i (\bbx^*)\right) \Big].
\end{align}
Compute the norm of both sides in \eqref{basic_lemma_proof_200}, and use the Cauchy-Schwarz inequality to obtain
\begin{align}\label{proof120}
\|\bbx^{k+1}-\bbx^*\|
\leq \frac{1}{n}\sum_{i=1}^n \left\| \bby_i^k-\bbx^* -\eps \left(\nabla f_i (\bby_i^k)-\nabla f_i (\bbx^*)\right)   \right\|.
\end{align}
Now we proceed to derive an upper bound for each summand in \eqref{proof120}. It can be shown that $\nabla f_i (\bby_i^k)-\nabla f_i (\bbx^*) = \nabla^2 f_i (\bbu_i^k)(\bby_i^k-\bbx^*)$ where $\bbu_i^k$ is a convex combination of $\bby_i^k$ and $\bbx^*$. Therefore, 
\begin{align}\label{proof12111}
\left\| \bby_i^k-\bbx^* -\eps \left(\nabla f_i (\bby_i^k)-\nabla f_i (\bbx^*)\right)   \right\|
 =\left\| \left(\bbI -\eps \nabla^2 f_i (\bbu_i^k)\right)(\bby_i^k-\bbx^*)   \right\| .
\end{align}
Since the functions $f_i$ are $\mu$-strongly convex and their gradients are $L$-Lipschitz continuous we can show that 
\begin{align}\label{proof12sda0}
\left\| \bby_i^k-\bbx^* -\eps \left(\nabla f_i (\bby_i^k)-\nabla f_i (\bbx^*)\right)   \right\| \leq \max\{|1-\eps\mu|,|1-\eps L|\} \| \bby_i^k-\bbx^*\|.
\end{align}
By setting the stepsize $\eps$ in \eqref{proof12sda0} as $\eps=2/(\mu+L)$, we can write
\begin{align}\label{proof12sda0uihubh}
\left\| \bby_i^k-\bbx^* -\eps \left(\nabla f_i (\bby_i^k)-\nabla f_i (\bbx^*)\right)   \right\| \leq \frac{\kappa-1}{\kappa+1}\ \| \bby_i^k-\bbx^*\|,
\end{align}
where $\kappa=L/\mu$ is the function $f_i$ condition number. By replacing the summands in the right hand side of \eqref{proof120} with their upper bounds $ (({\kappa-1})/({\kappa+1}))\| \bby_i^k-\bbx^*\|$, as shown in \eqref{proof12sda0uihubh}, we can show that the residual $\|\bbx^{k+1}-\bbx^*\|$ is bounded above as
\begin{align}\label{proof140}
\|\bbx^{k+1}-\bbx^*\|
 \leq\left( \frac{\kappa-1}{\kappa+1}\right) \sum_{i=1}^n \frac{\| \bby_i^k-\bbx^*\|}{n}\end{align}
 Note that in the DIAG method we use a cyclic scheme to update the variables. Thus, the set of variables $\{\bby_1^k,\dots,\bby_n^k\}$ is identical to the set of the last $n $ iterates before the iterate $\bbx^{k+1}$ which is given by $\{\bbx^k,\dots, \bbx^{k-n+1}\}$. Thus, we can replace the sum in $\sum_{i=1}^n \| \bby_i^k-\bbx^*\|$ in \eqref{proof140} by the sum $\sum_{i=1}^n \| \bbx^{k-i+1}-\bbx^*\|$ and the claim in \eqref{fundamental_lemma_claim} follows.

\section{Proof of Proposition \ref{CIAG_better_than_GD_lemma}}\label{apx_CIAG_better_than_GD_lemma}

Consider the definition of the constant $\rho:=(\kappa-1)/(\kappa+1)$ where $\kappa=L/\mu$ is the objective function condition number. Thus, if all the copies $\bby_i$ are initialized at $\bbx^0$, the result in Lemma \ref{fundamental_lemma} implies that
\begin{equation}\label{proof_rate_100}
\|\bbx^{1}-\bbx^*\|  \ \leq\  \rho\ \|\bbx^{0}-\bbx^*\|.
\end{equation}
We can use the same inequality for the second iterate to obtain
\begin{align}\label{proof_rate_200}
\|\bbx^{2}-\bbx^*\|  \leq\  \frac{\rho}{n}\ \|\bbx^{1}-\bbx^*\| + \frac{\rho(n-1)}{n}  \|\bbx^{0}-\bbx^*\| \end{align}
Replace $\|\bbx^{1}-\bbx^*\| $ in \eqref{proof_rate_200} by its upper bound in \eqref{proof_rate_100} and regroup the terms to obtain 
\begin{align}\label{proof_rate_201}
\|\bbx^{2}-\bbx^*\| \ 
& \leq\  \frac{\rho^2}{n}\ \|\bbx^{0}-\bbx^*\| + \frac{\rho(n-1)}{n}  \|\bbx^{0}-\bbx^*\| \nonumber\\
& =\  \rho \left[ 1- \frac{1-\rho}{n} \right] \|\bbx^{0}-\bbx^*\|.
\end{align}
Repeat the same process for the third residual $\|\bbx^{3}-\bbx^*\| $ to obtain
\begin{align}\label{proof_rate_300}
\|\bbx^{3}-\bbx^*\| \ & \leq\  \frac{\rho}{n}\ \|\bbx^{2}-\bbx^*\|+\frac{\rho}{n}\ \|\bbx^{1}-\bbx^*\| + \frac{\rho(n-2)}{n}  \|\bbx^{0}-\bbx^*\| \nonumber\\
& \leq\     \rho \left[1-\frac{2(1-\rho)}{n}-\frac{\rho(1-\rho)}{n^2}\right]\|\bbx^{0}-\bbx^*\| ,
\end{align}
where in the second inequality we use the bounds in \eqref{proof_rate_100} and \eqref{proof_rate_200}. Since the term $-\rho(1-\rho)/n^2$ is negative we can drop this term and show that the residual $\|\bbx^{3}-\bbx^*\| $ is upper bounded by 
\begin{align}\label{proof_rate_310}
\|\bbx^{3}-\bbx^*\| \  \leq\   \rho \left[1-\frac{2(1-\rho)}{n}\right]\|\bbx^{0}-\bbx^*\| .
\end{align}
By following the same logic we can show that for the first $n$ residuals $\{\|\bbx^k-\bbx^*\|\}_{k=1}^n$ the following inequality holds
\begin{align}\label{proof_rate_500}
\|\bbx^{k}-\bbx^*\| &\leq  \rho \left[1-\frac{(k-1)(1-\rho)}{n}\right]\|\bbx^{0}-\bbx^*\| ,\quad \for \ k=1,\dots,n.
\end{align}
Thus, the result in \eqref{eq_lin_result_1} holds.

Now we proceed to show that the result in  \eqref{eq_lin_result_2} hold for $k=n+1,\dots,2n$. According to the result in Lemma \ref{fundamental_lemma} we can write
\begin{align}\label{proof_new_result_100}
\|\bbx^{n+1}-\bbx^*\|   \leq\  \frac{\rho}{n}\ \|\bbx^{n}-\bbx^*\|+\dots+ \frac{\rho}{n}  \|\bbx^{1}-\bbx^*\| 
\end{align}
By replacing each summand in the right hand side of \eqref{proof_new_result_100} by its upper bound in \eqref{proof_rate_500} we obtain 
\begin{align}\label{proof_new_result_200}
\|\bbx^{n+1}-\bbx^*\|  & \leq\  \frac{\rho}{n}\ 
\left[ \sum_{i=1}^n \rho \left[1-\frac{(i-1)(1-\rho)}{n}\right]\right]\|\bbx^{0}-\bbx^*\|  
\nonumber\\
&
={\rho^2}{}\ 
\left[ 1-\frac{(1-\rho)(n-1)}{2n}  \right]\|\bbx^{0}-\bbx^*\| .
\end{align}
It follows from the result in \eqref{proof_new_result_200} that the residual $\|\bbx^{n+1}-\bbx^*\| $ is upper bounded by 
\begin{align}\label{proof_new_result_300}
\|\bbx^{n+1}-\bbx^*\|  & \leq\
{\rho^2}{} 
\left[ 1-\frac{1-\rho}{n} \times \min\left\{1,\frac{(n-1)}{2n} \right\} \right]\|\bbx^{0}-\bbx^*\| .
\end{align}
Therefore, the result in \eqref{eq_lin_result_2} holds for $k=n+1$. To prove the claim for $k=n+2$, first note that based on the upper bounds in \eqref{proof_rate_500} we can show that 
\begin{align}\label{proof_new_result_310}
\|\bbx^{k}-\bbx^*\| &\leq  \rho \left[1-\frac{(1-\rho)}{n}\right]\|\bbx^{0}-\bbx^*\| ,\quad \for \ k=2,\dots,n.
\end{align}
Considering the following inequality 
\begin{align}\label{proof_new_result_320}
  \left[ 1- \frac{1-\rho}{n} \right] \leq
  \left[ 1-\frac{1-\rho}{n} \times \min\left\{1,\frac{(n-1)}{2n} \right\} \right], 
\end{align}
and the upper bounds in \eqref{proof_new_result_310} we obtain that 
\begin{align}\label{proof_new_result_330}
\|\bbx^{k}-\bbx^*\| 
\leq 
{\rho} \left[ 1-\frac{1-\rho}{n} \times \min\left\{1,\frac{(n-1)}{2n} \right\} \right]\|\bbx^{0}-\bbx^*\|, \quad \for \ k=2,\dots,n.
\end{align}
In addition, the result in \eqref{proof_new_result_300} and the fact that $\rho<1$ imply that the term $\|\bbx^{n+1}-\bbx^*\| $ is also can be upper bounded by
\begin{align}\label{proof_new_result_340}
\|\bbx^{n+1}-\bbx^*\|  & \leq\
{\rho}{} 
\left[ 1-\frac{1-\rho}{n} \times \min\left\{1,\frac{(n-1)}{2n} \right\} \right]\|\bbx^{0}-\bbx^*\|.
\end{align}
Now, based on the result in Lemma \ref{fundamental_lemma}, we can write
\begin{align}\label{proof_new_result_350}
\|\bbx^{n+2}-\bbx^*\| \leq 
\frac{\rho}{n}\left[ \|\bbx^{n+1}-\bbx^*\|+\dots+ \|\bbx^{2}-\bbx^*\|\right]  .
\end{align}
The inequalities in \eqref{proof_new_result_330} and \eqref{proof_new_result_340} show that all the summands in \eqref{proof_new_result_350} are bounded by the same upper bound. Replace these term by the upper bound to obtain
\begin{align}\label{proof_new_result_800}
\|\bbx^{n+2}-\bbx^*\| 
\leq 
{\rho^2} \left[ 1-\frac{1-\rho}{n} \times \min\left\{1,\frac{(n-1)}{2n} \right\} \right]\|\bbx^{0}-\bbx^*\|, 
\end{align}
and the claim in \eqref{eq_lin_result_2} for $k=n+2$ follows. 

Since the constant $\rho$ is strictly less than 1, we can replace the upper bound in \eqref{proof_new_result_800} by 
\begin{align}\label{proof_new_result_900}
\|\bbx^{n+2}-\bbx^*\| 
\leq 
{\rho} \left[ 1-\frac{1-\rho}{n} \times \min\left\{1,\frac{(n-1)}{2n} \right\} \right]\|\bbx^{0}-\bbx^*\|.
\end{align}
Using the same argument, we can say that the upper bound in \eqref{proof_new_result_900} holds for $k=n+2,\dots,3$ and apply the result in Lemma \ref{fundamental_lemma} to show that the distance $\|\bbx^{n+3}-\bbx^*\| $ is bounded above by 
\begin{align}\label{proof_new_result_1000}
\|\bbx^{n+3}-\bbx^*\| 
\leq 
{\rho}^2 \left[ 1-\frac{1-\rho}{n} \times \min\left\{1,\frac{(n-1)}{2n} \right\} \right]\|\bbx^{0}-\bbx^*\|, 
\end{align}
which yields the claim in \eqref{eq_lin_result_2} for $k=n+3$. By repeating the steps in \eqref{proof_new_result_900} and \eqref{proof_new_result_1000} we can conclude that the result in \eqref{eq_lin_result_2} holds for $k=n+1,\dots,2n$.

The proof for steps $k>2n$ is similar to the argument used for the steps $k=n+1,\dots,2n$, although we write it in a formal manner by using induction. 

 Assume that for for $k=nj+1,\dots,nj+n$ the following inequality holds
\begin{align}\label{proof_rate_501}
\|\bbx^{k}-\bbx^*\| \leq  \rho^{\lfloor \frac{k-1}{n}\rfloor +1} \left[1-\frac{(1-\rho)}{n} \times \min\left\{1,\frac{n-1}{2}\right\}\right]\|\bbx^{0}-\bbx^*\| .
\end{align}
We intend to prove the same inequalities hold for $k=n({j+1})+1, \dots, n({j+1})+n$. Note that the result in \eqref{proof_rate_501} is satisfied for $j=1$, which corresponds to the iterates $k=n+1,\dots, 2n$, and the base of induction holds.

As we assume that the result in \eqref{proof_rate_501} holds for $k=nj+1,\dots,nj+n$, we obtain that 
\begin{align}\label{proof_rate_502}
\|\bbx^{k}-\bbx^*\| \leq  \rho^{j+1} \left[1-\frac{(1-\rho)}{n} \times \min\left\{1,\frac{n-1}{2}\right\}\right]\|\bbx^{0}-\bbx^*\| ,
\end{align}
for $k=nj+1,\dots,nj+n$. According to Lemma \ref{fundamental_lemma}, the residual $\|\bbx^{n(j+1)+1}-\bbx^*\|$ is bounded above by 
\begin{align}\label{proof_rate_503}
\|\bbx^{n(j+1)+1}-\bbx^*\|   \leq\  \frac{\rho}{n}\ \|\bbx^{nj+1}-\bbx^*\|+\dots+ \frac{\rho}{n}  \|\bbx^{nj+n}-\bbx^*\|.
\end{align}
Replacing the summands in the right hand side of \eqref{proof_rate_503} by their upper bound in \eqref{proof_rate_502} implies that 
\begin{align}\label{proof_rate_504}
\|\bbx^{n(j+1)+1}-\bbx^*\|   \leq\  \rho^{j+2} \left[1-\frac{(1-\rho)}{n} \times \min\left\{1,\frac{n-1}{2}\right\}\right]\|\bbx^{0}-\bbx^*\|, 
\end{align}
which yields the claim in \eqref{proof_rate_501} for $k=n({j+1})+1$. 

Now use the result in \eqref{proof_rate_504} and the inequality $\rho<1$ to obtain 
\begin{align}\label{proof_rate_505}
\|\bbx^{n(j+1)+1}-\bbx^*\|   \leq\  \rho^{j+1} \left[1-\frac{(1-\rho)}{n} \times \min\left\{1,\frac{n-1}{2}\right\}\right]\|\bbx^{0}-\bbx^*\|, 
\end{align}
Proceed by writing the result in Lemma \ref{fundamental_lemma} for $k=n(j+1)+2$ to obtain 
\begin{align}\label{proof_rate_506}
\|\bbx^{n(j+1)+2}-\bbx^*\|   \leq\  \frac{\rho}{n}\ \|\bbx^{nj+2}-\bbx^*\|+\dots+ \frac{\rho}{n}  \|\bbx^{nj+n+1}-\bbx^*\|.
\end{align}
Replace the summands in the right hand side of \eqref{proof_rate_506} by their upper bounds in \eqref{proof_rate_502} and \eqref{proof_rate_505}, which are the same upper bounds, to obtain 
\begin{align}\label{proof_rate_507}
\|\bbx^{n(j+1)+2}-\bbx^*\|   \leq\  \rho^{j+2} \left[1-\frac{(1-\rho)}{n} \times \min\left\{1,\frac{n-1}{2}\right\}\right]\|\bbx^{0}-\bbx^*\|, 
\end{align}
and the claim in \eqref{proof_rate_501} for $k=n({j+1})+2$ follows. By repeating the steps from \eqref{proof_rate_505} to \eqref{proof_rate_507} we can show that the same result holds for $k=k=n({j+1})+3,\dots,k=n({j+1})+n$. Thus, the inequality in \eqref{proof_rate_501} holds for $k=n({j+1})+1,\dots,n({j+1})+n$. The induction is complete which implies that the claim in \eqref{eq_lin_result_2} holds.

\section{Proof of Theorem \ref{thm_lin}}\label{apx_linear_convg_lemma}

Considering the result in \eqref{eq_lin_result_1} and the definition of the constant $a$ in \eqref{cond1} we obtain that  
\begin{align}\label{proof_lin_cnvg_200}
\|\bbx^{k}-\bbx^*\| &\leq  a \gamma^k\|\bbx^{0}-\bbx^*\| , \qquad \for \quad k=1,\dots,n.
\end{align}
Thus, the inequality in \eqref{lin_IAG} holds for steps $k=1,\dots,n$.

Now we proceed to show that the claim in  \eqref{lin_IAG} also holds for $k>n$. To do so, we use an induction argument. Let's assume we aim to show that the inequality in \eqref{lin_IAG} holds for $k=j$, while it holds for the last $n$ iterates $k=j-1,\dots, j-n$. According to the result in Lemma \ref{fundamental_lemma} we can write
\begin{align}\label{proof_lin_cnvg_300}
\|\bbx^{j}-\bbx^*\|\leq \rho \left[ \frac{\| \bbx^{j-1}-\bbx^*\|+\dots +\| \bbx^{j-n}-\bbx^*\|}{n}\right],
\end{align}
where $\rho=(\kappa-1)/(\kappa+1)$. Based on the induction assumption, for steps $k=j-1,\dots, j-n$, the result in \eqref{lin_IAG} holds. Thus, we can replace the terms in the right hand side of \eqref{proof_lin_cnvg_300} by the upper bounds from  \eqref{lin_IAG}. This substitution implies 
\begin{align}\label{proof_lin_cnvg_400}
\|\bbx^{j}-\bbx^*\| &\leq \frac{\rho a}{n} \left[\gamma^{j-1}+\dots+\gamma^{j-n} \right] \|\bbx^0-\bbx^*\|
\nonumber\\ &
=\frac{\rho a\gamma^{j-n}(1- \gamma^{n})}{n(1-\gamma)} \|\bbx^0-\bbx^*\|
\end{align}
Rearranging the terms in \eqref{cond2} allows us to show that $(\rho(1-\gamma^n))/(n(1-\gamma))$ is bounded above by $\gamma^n$. This is true since
\begin{align}\label{proof_lin_cnvg_500}
 \gamma^{n+1}-\left(1+\frac{\rho}{n}\right)\gamma^n +\frac{\rho}{n}\leq 0 
 &\iff
\frac{\rho}{n}(1-\gamma^n)-\gamma^n(1-\gamma)\leq 0\nonumber\\
&\iff \frac{\rho(1-\gamma^n)}{n(1-\gamma)}\leq \gamma^n.
\end{align}
Therefore, we can replace the term $(\rho(1-\gamma^n))/(n(1-\gamma))$ in \eqref{proof_lin_cnvg_400} by its upper bound $\gamma^n$ to obtain
\begin{align}\label{proof_lin_cnvg_600}
\|\bbx^{j}-\bbx^*\| &\leq
a\gamma^{j} \|\bbx^0-\bbx^*\|.
\end{align}
The result in \eqref{proof_lin_cnvg_600} completes the proof. Thus, by induction the claim in  \eqref{lin_IAG}  holds for all $k\geq 1$ if the conditions in \eqref{cond1} and \eqref{cond2} satisfied.

\section{Proof of Proposition \ref{prop_non_empty_set}}\label{apx_prop_non_empty_set}

To prove the claim in Proposition \ref{prop_non_empty_set} we first derive the following lemma.

\begin{lemma}\label{lemma_aux}
For all $n\geq 1$ and $0\leq \phi\leq 1$ we have
\begin{align}\label{eq_aux_lemma}
 \left(1-\frac{\phi}{n} \right)^n\leq  \left(1-\frac{\phi}{n+1} \right)^{n+1}
\end{align}
\end{lemma}

\begin{proof}
Consider the function $ h(x)=(1-({\phi}/{x}))^x$  for $x>1$. The natural logarithm of the function $h(x)$ is given by $
\ln \left(h(x)\right)= x\ln(1-({\phi}/{x})) $. Compute the derivative of both sides with respect to $x$ to obtain
\begin{equation}
\frac{dh}{dx} \times \frac{1}{h(x)}= \ln\left(1-\frac{\phi}{x}\right)+ x\times \frac{\frac{\phi}{x^2}}{1-\frac{\phi}{x}}
\end{equation}
By multiplying both sides by $h(x)$, replacing $h(x)$ by the expression $\left(1-\frac{\phi}{x} \right)^x$, and simplifying the terms we obtain that the derivative of the function $h(x)$ is given by 
\begin{equation}\label{papa}
\frac{dh}{dx}
 =\left(1-\frac{\phi}{x} \right)^x \left[\ln\left(1-\frac{\phi}{x}\right)
 	+ \frac{\frac{\phi}{x}}{1-\frac{\phi}{x}}\right].
\end{equation}
Note that the sum $\ln(1-u)+u/(1-u)$ is always positive for $0<u<1$. By setting $u:=\phi/x$, we can conclude that the term in the right hand side of \eqref{papa} is positive for $x>1$. Therefore, the derivative $dh/dx$ is always positive for $x>1$. Thus, the function $h(x)$ is an increasing function for $x>1$ and we can write 
\begin{align}
 \left(1-\frac{\phi}{n} \right)^n\leq  \left(1-\frac{\phi}{n+1} \right)^{n+1},
\end{align}
for $0\leq\phi\leq1$ and $n>1$.
It remains to show that the same claim is also valid for $n=1$ which is equivalent to the inequality 
\begin{align}\label{hamin}
1-\phi \leq  \left(1-\frac{\phi}{2} \right)^{2}.
\end{align}
It is trivial to show \eqref{hamin} holds, and, therefore, the claim in \eqref{eq_aux_lemma} holds for all $n\geq1$.
\end{proof}

Now proceed to prove the claim in Proposition \ref{prop_non_empty_set} using the result in Lemma \ref{lemma_aux}. To prove that the feasible set of the condition in \eqref{cond2} is non-empty we show that $\gamma={\rho}^{1/n}$ satisfies the inequality in \eqref{cond2}. In other words, 
\begin{equation}\label{eq_proof_non_empty_100}
\rho^{\frac{n+1}{n}}-\left(1+\frac{\rho}{n}\right)\rho+\frac{\rho}{n}\leq 0
\end{equation}
Divide both sides of \eqref{eq_proof_non_empty_100} by $\rho$ and regroupe the terms to obtain the following ineqaulity
\begin{equation}\label{eq_proof_non_empty_200}
\rho\leq  \left(1-\frac{1-\rho}{n} \right)^n,
\end{equation}
which is equivalent to \eqref{eq_proof_non_empty_100}. In other words, the inequality in \eqref{eq_proof_non_empty_200} is a necessary and sufficient condition for the condition in  \eqref{eq_proof_non_empty_100}. 

Recall the result in Lemma \ref{lemma_aux}. By setting $\phi=1-\rho$ we obtain that 
\begin{equation}\label{eq_proof_non_empty_300}
\rho = \left(1-\frac{1-\rho}{1} \right)^1\leq  \left(1-\frac{1-\rho}{2} \right)^2 \leq \dots \leq  \left(1-\frac{1-\rho}{n} \right)^{n},  
\end{equation}
for $n\geq1$. Thus, the inequality in \eqref{eq_proof_non_empty_200} holds, and, consequently, the inequality in \eqref{eq_proof_non_empty_100} is valid. Therefore, $\gamma={\rho}^{1/n}$ satisfies the inequality in \eqref{cond2}.

Then, we can define $a$ as the smallest constant that satisfies \eqref{cond1} for the choice $\gamma={\rho}^{1/n}$, which is given by 
\begin{align}\label{eq_proof_non_empty_400}
a= \max_{k=1,\dots,n} \left(1-\frac{(k-1)(1-\rho)}{n}\right) \rho^{1-\frac{k}{n}}.
\end{align}
Therefore, $\gamma={\rho}^{1/n}$ and the constant $a$ in \eqref{eq_proof_non_empty_400} satisfy the conditions in \eqref{cond1} and \eqref{cond2}, and the claim in Proposition \ref{prop_non_empty_set} follows.

\section{Proof of Theorem \ref{theorem-rate-bounds}}\label{apx_theorem-rate-bounds}

\begin{itemize}
	\item [$(i)$] This follows directly from  the Perron-Frobenius theorem for irreducible non-negative matrices \cite[Theorem 8.4.4]{HorJoh91}. 
	\item [$(ii)$] By \cite[Theorem 8.5.1]{HorJoh91}, we also have
  \begin{equation}  
   	\lim_{k \to \infty} \frac{\bbM_\rho^k}{\lambda^*(\rho)^k} = \bbu\bbv^T 
   	\label{eigen-vec-limit}
   \end{equation} 
where $\bbu$, $\bbv$ are the right and left eigenvectors of $\bbM_\rho$ corresponding to the eigenvalue $\lambda^*(\rho)$ normalized to satisfy $\bbv^T \bbu = 1$. Note also that 
 \begin{equation}
d^{k} = \bbe_1^T \bbd^{k+1} = \bbe_1^T \bbM_\rho^{k-n+1} \bbd^n
\end{equation}
Therefore, 
\begin{align}
\lim_{k\to\infty} d^{k+1}/d^k = \lim_{k\to\infty} \frac{\bbe_1^T \bbM_\rho^{k-n+2} \bbd^n}{\bbe_1^T \bbM_\rho^{k-n+1} \bbd^n} &= \lim_{k\to\infty}  \frac{\bbe_1^T \bbM_\rho^{k-n+1} \bbd^n / {\lambda^*(\rho)}^{k-n+1} }{\bbe_1^T \bbM_\rho^{k-n} \bbd^n / {\lambda^*(\rho)}^{k-n+1} } 
\end{align}
where in the last inequality we divided both numerator and denominator with the same factor $\lambda^*(\rho)^{k-n+1}$. Simplifying and using \eqref{eigen-vec-limit}, we obtain
 \begin{align}
\lim_{k\to\infty} d^{k+1}/d^k =  \lim_{k\to\infty} \lambda^*(\rho) \frac{\bbe_1^T \bbu\bbv^T \bbd^n }{\bbe_1^T \bbu\bbv^T \bbd^n } =  \lambda^*(\rho).
\end{align}
	\item [$(iii)$] A direct consequence of [Theorem 8.1.22]\cite{HorJoh91} is that
			$\lambda^*(\rho) \geq \rho$, this proves the lower bound on 
	$\lambda^*(\rho)$. To get the upper bound, let $ \textbf{1} = [1 ~ 1~  ~1 \dots 1]^T$ 
	be the vector of ones.	We will show that 
			\begin{equation}
				 \bbM_\rho^{2n} \textbf{1} < \rho^2 \textbf{1}
				 \label{componentwise_ineq}
		   \end{equation}		 
	where the notation $``<"$ denotes the componentwise inequality for vectors. Then, by \cite[Corollary 8.1.29]{HorJoh91}, this would imply 
	         $$\lambda^*(\rho)^{2n} < \rho^2$$
	which is equivalent to the desired upper bound. It is a straightforward computation to show that if we set $\bbd^n = \textbf{1}$, then after a simple induction argument we obtain $d^{n} = \rho$ and $d^{n+1} < \rho$,  $d^{n+2} < \rho$, \dots, $d^{2n-1}<\rho$, i.e. 
\begin{equation}
 \bbd^{2n}	=	\begin{bmatrix}
    d^{2n-1}        \\
    d^{2n-2} \\
    \hdots \\
    d^{n}
\end{bmatrix} = \bbM_\rho^{n} \bbd^n = \bbM_\rho^{n} \textbf{1} = \rho \bbv 
\end{equation}
for a vector $\bbv = [v_1, v_2, \dots, v_n]^T$, where $v_i < 1$ if $i<n$ and $v_n = 1$. Using similar arguments we can write 
\begin{equation}
\bbd^{3n}	=	\begin{bmatrix}
    d^{3n-1}        \\
    d^{3n-2} \\
    \hdots \\
    d^{2n}
\end{bmatrix} = \bbM_\rho^{2n} \bbd^n = \bbM_\rho \bbM_\rho^{n} \textbf{1} = \rho \bbM_\rho \bbv 
\end{equation}
and a straightforward computation shows that $\bbM_\rho \bbv < \rho \textbf{1}.$ Combining this inequality with the previous equation proves \eqref{componentwise_ineq} and concludes the proof.
	\item [$(iv)$] It follows from \eqref{char_polynomial} that the roots of the polynomial $h$ is the same as the roots of the polynomial $T$ except that $h$ has an additional root at 1. By part $(i)$ and $(iii)$, $\lambda^*(\rho)$ is the largest real root of $T$ and $0<\lambda^*(\rho)<1$. Therefore, it is also the largest real root of the function $h$ over the interval $(0,1)$ which is equal to $\gamma_0$.   
\end{itemize}


\bibliographystyle{siamplain}
\bibliography{bmc_article,bmc_article2,bibliography,bibliography2}
\end{document}